\documentclass[article,ij4uq]{ij4uq}              %  final version
\usepackage[hang]{footmisc}
\usepackage[ruled]{algorithm2e}
\usepackage{algpseudocode}

\usepackage{epstopdf}

\fancypagestyle{plain}{%
	\fancyhf{}
	\fancyhead[R]{\small {\it International Journal for Uncertainty Quantification}, XXX}
	\fancyfoot[R]{\small\bf\thepage }
	\fancyfoot[L]{\fottitle www.begellhouse.com}
}

\renewcommand{\dmyy}{19}
\renewcommand{\myyear}{2020}
\renewcommand{\today}{}

\newcommand{\bbE}{\mathbb{E}}
\newcommand{\bbP}{\mathbb{P}}
\newcommand{\bbQ}{\mathbb{Q}}
\newcommand{\bbZ}{\mathbb{Z}}
\newcommand{\bbN}{\mathbb{N}}

\newcommand{\cO}{\mathcal{O}}

\newcommand{\cI}{{\cal I}}

\def\bbR{\mathbb{R}}

\newtheorem{ass}{Assumption} 

\theoremstyle{plain}
\newtheorem{lem}{Lemma}

\theoremstyle{remark}
\newtheorem{exam}{Example}

\def\blu#1{{{#1}}}

\begin{document}

\volume{Volume X, Issue XXX, \myyear\today}
\title{Multi-Index Sequential Monte Carlo Methods for partially observed Stochastic Partial Differential Equations}
\titlehead{MISMC$^2$ for SPDE}
\authorhead{Jasra, Law, \& Xu}

\author[1]{Ajay Jasra}
\email{ajay.jasra@kaust.edu.sa}
\corrauthor[2]{Kody J. H. Law}
\corremail{kody.law@manchester.ac.uk}
\corrurl{https://sites.google.com/view/kodylaw/home}
\author[1]{Yaxian Xu}
\email{staxy@nus.edu.sg}
\address[1]{Computer, Electrical and Mathematical Science and Engineering Division, King Abdullah University of Science and Technology, Thuwal, 23955-6900, KSA.}
\address[2]{Computer Science and Mathematics Division, Oak Ridge National Laboratory, Oak Ridge, TN, USA, 37831}
\address[3]{Department of Statistics \& Applied Probability National University of Singapore Singapore, Singapore}

\dataO{\mydate\today}

\dataF{\mydate\today}

%\begin{abstract}
\abstract{In this paper we consider sequential joint state and static parameter estimation given %for
discrete time observations associated to a partially observed stochastic
partial differential equation. % (SPDE). % as data arrive sequentially.
It is assumed that one can only estimate the hidden state %, recursively in time,
using a discretization of the model.
%and further that this discretization can only be fit using advanced Monte Carlo methods such as sequential Monte Carlo (SMC).
In this context, it is known that the multi-index Monte Carlo (MIMC)
method of \cite{mimc} can be used to improve over direct Monte Carlo
from the most precise discretizaton. However, in the context of interest, it cannot be directly applied, but rather must be  used within another %advanced
method such as sequential Monte Carlo (SMC).
We show how one can
use the MIMC method by renormalizing the %using a re-normalized
standard identity and approximating the resulting identity using the SMC$^{2}$ method of \cite{chopin2}, {which is %one of the only existing
an exact method that can be used in this context}.
We prove %, under assumptions,
that our approach can %improve, in terms of reduced
reduce the cost to obtain %for
a given mean square error, % (MSE),
relative to just using SMC$^{2}$ on the most precise
discretization. We demonstrate this with some numerical examples.}
%\\
%\textbf{Key Words}:
%\end{abstract}

\maketitle

\section{Introduction}

We consider joint state and static parameter estimation for discrete time observations, associated to a partially observed stochastic
partial differential equation (SPDE).
Such models can be considered a form of hidden Markov model (HMM),
and these have a significant number of practical applications; see e.g.~\cite{Cappe_2005} for instance. {See Figure \ref{fig:hmm} for a schematic.
The objective is to estimate the states
$(X_0,X_1,\dots, X_n) \in \mathsf{X}^{n+1}$ and parameters $\theta\in \Theta$
%given the using a discretized version of the SPDE, i.e.
given the data $(y_0,\dots,y_n)$, i.e.~we want to find
$$
\bbP(X_0\in A_0,X_1\in A_1,\dots, X_n\in A_n, \theta\in A_T |y_0,\dots,y_n)
$$
recursively in time $n$.}
%The data arrives sequentially, and therefore it is of interest to perform inference
%sequentially as well, using a recursive method, i.e.
%\[
%\begin{split}
%&\bbP(X_0\in A_0,X_1\in A_1,\dots, X_n\in A_n, \theta\in A_T | Y_1=y_1, \dots, Y_n=y_n) \mapsto \\
%&\bbP(X_0\in A_0,X_1\in A_1,\dots, X_n\in A_n, X_{n+1} \in A_{n+1}, \theta\in A_T
%| Y_1=y_1, \dots, Y_n=y_n, Y_{n+1}) \, .
%\end{split}
%\]}

\begin{figure}
	\centering
	\includegraphics[width=.7\columnwidth]{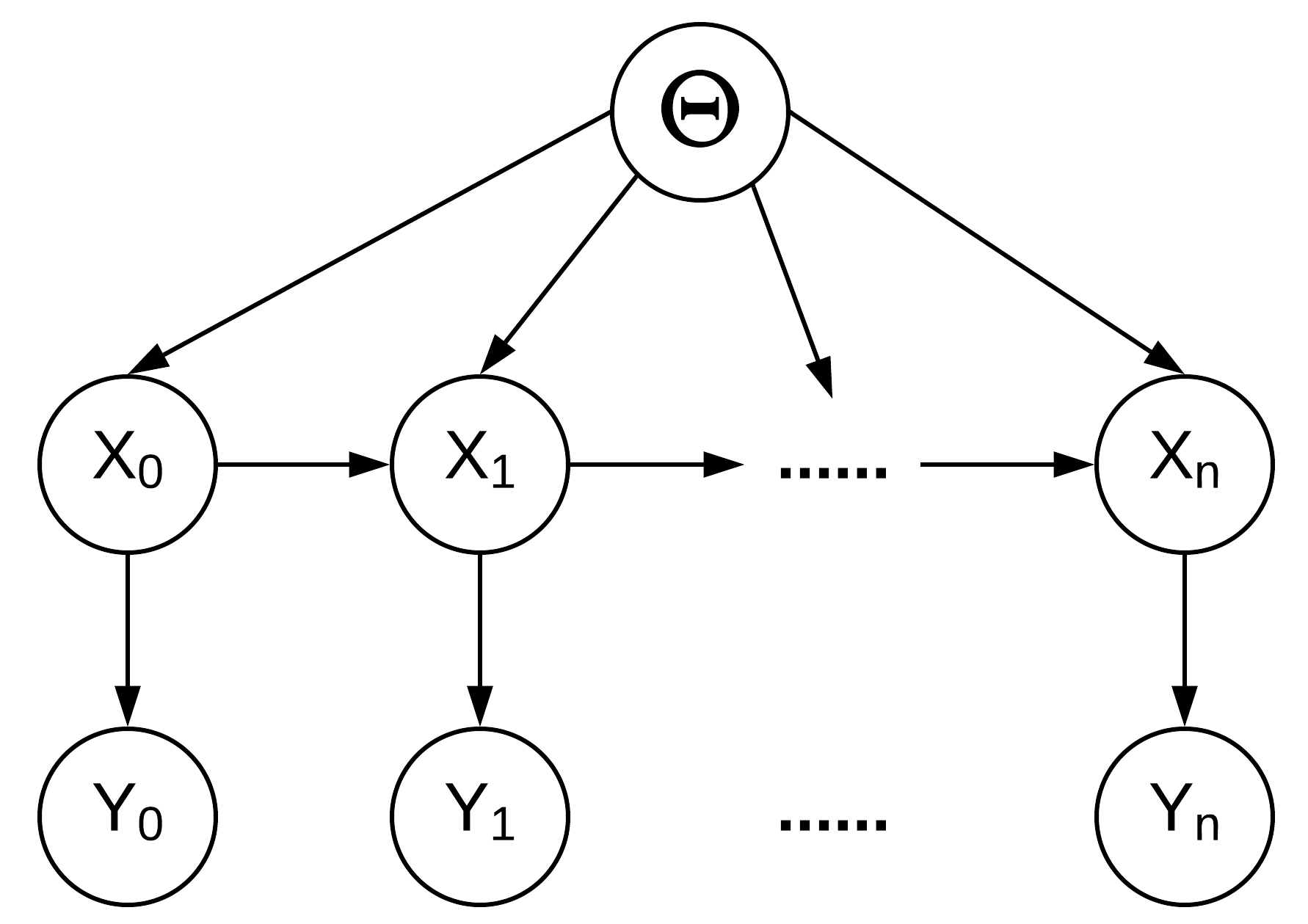}
	\caption{A graphical model of the HMM studied in this paper.}
	\label{fig:hmm}
\end{figure}

{In this article we focus upon the scenario where one will have to discretize the time and space element of the SPDE. %stochastic partial differential equation (SPDE).
}
%given a data model and prior on the static parameters,  }
%In the context of interest, we assume that one will have to discretize the time and space element of the SPDE and estimate the states
%and parameters using a discretized version of the model.
In this scenario, one is faced with the problem of %left with a
joint state and static parameter estimation for a HMM with
%,albeit for a (typically)
a high-dimensional state; %. This
a problem which is notoriously challenging. {The main issue is that for any fixed static parameter, one can seldom calculate the joint density (the smoother), given the data, of the hidden states over the observation times.
%The inclusion of an unknown parameter makes the %only exacerbates the issue,
Joint inference on the parameter is even more challenging,
due to the dependence of all the hidden states on the parameter
(see Figure \ref{fig:hmm}).
The state of the art method for consistently solving this problem for a fixed time $n$
is the particle Markov chain Monte Carlo (PMCMC) method \cite{andrieu}, which
involves using sequential Monte Carlo (SMC) within Markov chain Monte Carlo (MCMC).
SMC methods can approximate the sequence of joint distributions on $X_{0:n}$ for $\theta$ fixed (smoothers).
They consist of sampling $N$ samples (also called particles) in parallel, sequentially in time.
SMC methods involve the recursion of a mutation step,
an importance sampling step, and a resampling step.
They provide
consistent (as $N\rightarrow\infty$) approximations of expectations w.r.t.~the smoother.
In  many contexts, SMC is referred to as a particle filter,  hence the name PMCMC.
We remark, however, one also seeks to perform statistical inference sequentially in time,
which adds yet another complication.
There exists a methodology which can extend PMCMC to a dynamic context,
called SMC$^2$ \cite{chopin2}.
{This %is an approach which is essentially a class of
method is an SMC algorithm which %will
sequentially samples from
a sequence of auxiliary target distributions
(the targets of a certain PMCMC algorithm),}
which admit the distribution of interest as a marginal.
The main reason why one %wishes to
would consider such a level of complexity,
is that conventional SMC methodology
which is designed for joint state and parameter inference associated to HMMs suffers from the so-called path degeneracy issue (see e.g.~\cite{kantas1}), which
renders it useless for long time intervals.}

In the problem of interest, we are dealing with an expectation w.r.t.~a probability measure {which is defined on a high-dimensional continuous state-space as a result of the space and time discretization}. It is assumed that the computational cost associated to performing any simulation-based numerical method will increase as the precision of the discretization is enhanced.
{In such a context, the multilevel Monte Carlo (MLMC) method is
very effective for reducing the
cost to achieve a given level of accuracy \cite{giles,giles1,hein}.
Following the success of the MLMC method,
%multilevel Monte Carlo (MLMC) method for scalar discretizations \cite{giles,giles1,hein}, t
the work \cite{mimc} revisited the MLMC identity
through the lense of sparse grids for multi-dimensional discretizations.
This more general method, which reduces to MLMC for one dimensional (discretizations)  problems,
%which was developed
is called multi-index Monte Carlo (MIMC).
This method can be much more efficient than MLMC 
for higher dimensional problems with suitable regularity,
for example providing canonical complexity $\cO(\varepsilon^{-2})$
to achieve mean square error (MSE) $\cO(\varepsilon^2)$.}
%ident was developed.
%This is an approach which
In this approach one rewrites the expectation of interest %w.r.t.~the most precise possible (e.g.~computationally feasible)  distribution
as a sum of difference of differences (DOD) w.r.t.~independent refinements
of the discretization levels of different dimensions,
for example different spatial dimensions and time.
%increasingly
%less precise discretizations
%(details are given in Section \ref{sec:prob_form}).
If the number of dimensions which are discretized
%index (i.e.~the dimension of discretization)
is $d$, then this %s a
DOD involves %of
expectations w.r.t.~at most $2^d$ different discretization levels.
Then, under appropriate assumptions and given an efficient \emph{coupling} of these $2^d$ distributions, % and an appropriate grid for the discretization,  show that
the cost to achieve a prespecified MSE is reduced with respect to considering the Monte Carlo method using a single discretization level \cite{mimc}.
Indeed, under appropriate assumptions, {relying strongly on the mixed regularity of the solution of the SPDE, and with an appropriate index set, the MIMC approach can achieve a substantial improvement in  cost with respect to %the single index counterpart,
MLMC \cite{mimc}.} {The coupling is essential to ensuring that the \emph{variance}
of the estimates of the DODs decay appropriately w.r.t.~the discretization indices. This allows one to use fewer samples in
approximating higher index residuals, hence balancing the cost in an optimal way.}

{%As noted above and as is critical in MLMC,
The sampling of a %\emph{good}
``good coupling'' of the $2^d$ distributions is the main key to cost reduction
in MIMC, as in MLMC.
%{The main contribution of this article is to construct an appropriate coupling identity and approximation of such an identity to leverage upon the success of the MIMC method.
%This is non-trivial because, as %already
%stated above, t
The problem of approximating expectations w.r.t.~the joint distribution of interest for a single discretization level/index is already challenging, and coupling $2^d$ such distributions is naturally much more complex.}
%In \cite{jklz}, based upon an idea in \cite{jklz1} (see also \cite{franks}),
{A method for using MCMC
%in multi-index contexts
within the MIMC framework was developed in \cite{jklz},
based upon an idea in \cite{jklz1} developed originally for PMCMC (see also \cite{franks}).
%, based on
The method involves constructing an approximate coupling of the $2^d$
targets and using this to approximate a re-normalized multi-index (MI) identity.
See \cite{ml_rev} for a pedagogical introduction to this general approximate coupling strategy.
%This procedure is typically inefficient for time-dependent scenarios, in which one aims
In the present context, the aim is to perform inference sequentially as data arrives,
%In this article, we extend t
%The approach is extended to the sequential context of joint state and static parameter estimation for discrete time observations, associated to a partially observed
where the unobserved process is an SPDE. %stochastic partial differential equation. %,as data arrive sequentially.
%Our multi-index decomposition, is approximated using SMC$^{2}$.
%use the MIMC method by using a
%A re-normalized multi-index (MI) identity is approximated %and approximating it
This is done using the SMC$^{2}$ method described above.
\blu{It involves first
%the culmination of a few steps, which are emphasized here.
extending the MLMC PMCMC method of \cite{jklz1} to the MIMC context,
in order to accommodate an SPDE model. %, and then
Second,
this new MIMC PMCMC is deployed within an SM$\textrm{C}^2$ algorithm,
which yields the MISMC$^2$ algorithm. 
The latter generalization of course yields MLSMC$^2$ as a byproduct, 
since ML is a particular case of MI.} %We prove, u
Under appropriate assumptions, we prove
that our approach can %improve, in terms of
reduce the cost %for
required to obtain a given MSE, %mean square error,
relative to just using SMC$^{2}$ on the most precise
discretization, {or even using the MLMC version}.
%We
This is demonstrated with a %in
numerical example.

\blu{Before proceeding, and to avoid confusion, we briefly digress on the distinction from existing
applications of MLMC to SMC. For a more comprehensive discussion, 
the reader is referred to the recent review \cite{ml_rev}. 
In particular, in \cite{beskos} the authors employ an SMC algorithm over {\em levels} 
(i.e. level there is analogous to time in the present), 
and subsequently constructed MLMC estimators using importance sampling estimators of 
the level increments. 
In contrast, in \cite{jasra2017multilevel} the authors couple pairs of SMC algorithms
in the particle filtering context using a coupled resampling mechanism.}

This article is structured as follows. {In Section \ref{sec:cartoon}, we provide a high-level introduction into the underlying idea of the article.}
In Section \ref{sec:prob_form} we describe the problem and how it may be solved, if numerical approximation were not required.
In Section \ref{sec:method} we show how our approach can be numerically approximated. In Section \ref{sec:theory} the theoretical result is given, with the proofs in the appendix.
Numerical results are presented in Section \ref{sec:numerics} .

{
\section{High-Level Discussion of the Approach}
\label{sec:cartoon}

The method presented in this article is quite complicated. To assist non-experts, we provide
a high-level description of the basic idea with minimal notations, 
in which one dimension is discretized. 
First we explain the approximate coupling strategy used to enable MIMC estimation.
Then we explain the PMCMC and SMC$^2$ methods.

%when one has to discretize in one dimension and minimal notations.

\subsection{Approximate coupling}

Consider a probability density on state-space $\mathsf{X}$
$$
p(x) := \frac{J(x)F(x)}{Z}
$$
where $J,F$ are two positive, real-valued functions, $\int_{\mathsf{X}}F(x)dx=1$, 
and $Z=\int_{\mathsf{X}}J(x)F(x)dx$ (assumed to be finite). It is of interest to compute
expectation of real-valued functions $\varphi:\mathsf{X}\rightarrow\mathbb{R}$  that are $p-$integrable:
$$
\mathbb{E}_p[\varphi(X)] = \int_{\mathsf{X}}\varphi(x)p(x)dx.
$$
Suppose that one only %can only evaluate /
has access to a sequence $J_l(x)F_l(x)$, $l\in\{0,1,\dots\}$ which are positive, real-valued functions such that
$$
\lim_{l\rightarrow\infty}\mathbb{E}_{p_l}[\varphi(X)] = \mathbb{E}_p[\varphi(X)]
$$
where $\mathbb{E}_{p_l}[\varphi(X)] = \int_{\mathsf{X}}\varphi(x)p_l(x)dx$, $p_l(x)=[J_l(x)F_l(x)]/Z_l$, $Z_l=\int_{\mathsf{X}}J_l(x)F_l(x)dx$.

Now, the MLMC identity
$$
\mathbb{E}_{p_L}[\varphi(X)] = \sum_{l=1}^L\{\mathbb{E}_{p_l}[\varphi(X)]-\mathbb{E}_{p_{l-1}}[\varphi(X)]\} + \mathbb{E}_{p_0}[\varphi(X)]
$$
can be very useful to reduce the computational effort in the Monte Carlo approximation of $\mathbb{E}_{p_L}[\varphi(X)]$, to achieve a given error (versus considering only $\mathbb{E}_{p_L}[\varphi(X)]$). 
The key to this method is the ability %beingable 
to construct a coupling $\check{p}_{l,l-1}$ of $(p_l,p_{l-1})$ for each $l\in\{1,2\dots\}$, i.e. a
%That is,  
probability density function on $\mathsf{X}\times\mathsf{X}$ such that
for every $(x,{x'})\in\mathsf{X}\times\mathsf{X}$
\begin{eqnarray*}
p_l(x) & = & \int_{\mathsf{X}} \check{p}_{l,l-1}(x,{x'})d{x'} \\
p_{l-1}({x'}) & = & \int_{\mathsf{X}} \check{p}_{l,l-1}(x,{x'})dx.
\end{eqnarray*}
Then, one has
\begin{equation}\label{eq:increment}
\mathbb{E}_{p_l}[\varphi(X)]-\mathbb{E}_{p_{l-1}}[\varphi(X)] = \int_{\mathsf{X}\times\mathsf{X}}\varphi(x) \check{p}_{l,l-1}(x,{x'})d(x,{x'}) - \int_{\mathsf{X}\times\mathsf{X}}\varphi({x'}) \check{p}_{l,l-1}(x,{x'})d(x,{x'}).
\end{equation}
%The idea %then, 
%is that i
If the coupling is sufficiently good, so that for instance %that
$$
\int_{\mathsf{X}\times\mathsf{X}}(\varphi(x)-\varphi({x'}))^2 \check{p}_{l,l-1}(x,{x'})d(x,{x'})  \leq h(l) \, ,
$$
where $\lim_{l\rightarrow\infty}h(l)=0$, $h$ is a positive, real-valued, monotonically decreasing 
function on $\{0,1,\dots\}$, then the aforementioned benefits are possible; see e.g.~\cite{giles,giles1,hein}.
The Monte Carlo method would rely on exact sampling from the distribution associated to the coupling 
$\check{p}_{l,l-1}$.

Let $l\geq 1$ be fixed.
In many practical problems of interest, such as the one considered in this article, 
deriving a suitable coupling $\check{p}_{l,l-1}$ which is amenable to known simulation methodology can
be very challenging. 
The basic idea used in this paper
and as adopted in \cite{jklz1} is as follows. 
Suppose one can find a coupling
$\check{F}_{l,l-1}$ of $(F_l,F_{l-1})$
%That is, 
%i.e. a probability density function on $\mathsf{X}\times\mathsf{X}$ such that for every $(x,{x'})\in\mathsf{X}\times\mathsf{X}$
such that 
\begin{eqnarray}\label{eq:checkf}
F_l(x) & = & \int_{\mathsf{X}} \check{F}_{l,l-1}(x,{x'})d{x'} \\
F_{l-1}({x'}) & = & \int_{\mathsf{X}} \check{F}_{l,l-1}(x,{x'})dx \, ,
\nonumber
\end{eqnarray}
%It is supposed that $\check{F}_{l,l-1}$ is sufficiently good %for instance
%that
and
$$
\int_{\mathsf{X}\times\mathsf{X}}(\varphi(x)-\varphi({x'}))^2 \check{F}_{l,l-1}(x,{x'})d(x,{x'})  \leq h(l) \, .
$$
Now set
$$
\tilde{p}_{l,l-1}(x,{x'}) = \frac{\max\{J_l(x),J_{l-1}({x'})\}\check{F}_{l,l-1}(x,{x'})}{\tilde{Z}_{l,l-1}} \, ,
$$
with $\tilde{Z}_{l,l-1}=\int_{\mathsf{X}\times\mathsf{X}}\max\{J_l(x),J_{l-1}({x'})\}\check{F}_{l,l-1}(x,{x'})d(x,{x'})$ (assumed to be finite). %Now, we n
Note that
\begin{eqnarray}
\mathbb{E}_{p_l}[\varphi(X)] & = & \int_{\mathsf{X}}\varphi(x)p_l(x)dx \label{eq:marginal} \\ \nonumber
& = & \frac{1}{Z_l}\int_{\mathsf{X}\times\mathsf{X}}\varphi(x)J_l(x)\check{F}_{l,l-1}(x,{x'}) d(x,{x'}) \\ \nonumber
& = & \frac{\tilde{Z}_{l,l-1}}{Z_l} \int_{\mathsf{X}\times\mathsf{X}}\varphi(x)\frac{J_l(x)}{\max\{J_l(x),J_{l-1}({x'})\}}\tilde{p}_{l,l-1}(x,{x'}) d(x,{x'}) \\ \nonumber
& = &  \mathbb{E}_{\tilde{p}_{l,l-1}}\Big[\varphi(X)\frac{J_l(X)}{\max\{J_l(X),J_{l-1}({X'})\}}\Big]\Big/\mathbb{E}_{\tilde{p}_{l,l-1}}\Big[\frac{J_l(X)}{\max\{J_l(X),J_{l-1}({X'})\}}\Big]
\end{eqnarray}
where %expectations w.r.t.~$\tilde{p}_{l,l-1}$ are denoted $\mathbb{E}_{\tilde{p}_{l,l-1}}$ and t
\blu{the last line uses that 
$\mathbb{E}_{\tilde{p}_{l,l-1}}\Big[\frac{J_l(X)}{\max\{J_l(X),J_{l-1}({X'})\}}\Big] = {Z_l}/{\tilde{Z}_{l,l-1}}$.} 
Thus
\begin{eqnarray*}
\mathbb{E}_{p_l}[\varphi(X)]-\mathbb{E}_{p_{l-1}}[\varphi(X)] & = &
\mathbb{E}_{\tilde{p}_{l,l-1}}\Big[\varphi(X)\frac{J_l(X)}{\max\{J_l(X),J_{l-1}({X'})\}}\Big]\Big/\mathbb{E}_{\tilde{p}_{l,l-1}}\Big[\frac{J_l(X)}{\max\{J_l(X),J_{l-1}({X'})\}}\Big] - \\ & &
\mathbb{E}_{\tilde{p}_{l,l-1}}\Big[\varphi({X'})\frac{J_{l-1}({X'})}{\max\{J_l(X),J_{l-1}({X'})\}}\Big]\Big/\mathbb{E}_{\tilde{p}_{l,l-1}}\Big[\frac{J_{l-1}({X'})}{\max\{J_l(X),J_{l-1}({X'})\}}\Big] \, .
\end{eqnarray*}
The main interest of this identity is the fact that one %can sometimes 
may be able to construct a coupling like $\check{F}_{l,l-1}$ 
and very efficient sampling methods for $\tilde{p}_{l,l-1}$, whereas this may not be
the case for $\check{p}_{l,l-1}$. It is then possible (e.g.~\cite{jklz1}) that the benefits of the MLMC method 
can be achieved, even though one does not know how to sample from a good coupling 
$\check{p}_{l,l-1}$.
}

\subsection{Monte Carlo methods}

%INTUITIVE PARAGRAPH ON PMCMC AS A PSEUDO-MARGINAL+ METHOD ????
%ALSO, BELOW IS MAYBE NOT QUITE ACCURATE.
%

\blu{First a simplified description of PMCMC is given.
Suppose one aims to estimate expectations with respect to the joint state and 
parameter smoothing distribution associated to an HMM, as described above, 
i.e. the aim is to approximate
$$
p(x_{1:n}, \theta | y_{1:n}) \propto p(x_{1:n} | y_{1:n}, \theta) p(y_{1:n}|\theta) p(\theta) \, .
$$
The first factor on the right-hand side is the smoothing distribution for a fixed parameters,
which can be well-approximated consistently with a particle filter.
The second term is the marginal likelihood, for which an unbiased and non-negative estimator 
is available via the particle filter.
It seems reasonable then to consider approximating the target distribution using a particle filter.
The PMCMC method leverages this intuition by %replacing the target distribution above 
using a non-negative unbiased estimate of the un-normalized joint distribution above
derived from the particle filter within an MCMC method. 
Let $\bbP_n(v_{1:n} | \theta)$ denote the distribution of {\em all auxiliary variables} 
$v_{1:n} = (x_{1:n}^{1:N}, a_{1:n}^{1:N})$ of an $N-$particle filter targeting the smoothing distribution
$p(x_{1:n} | y_{1:n}, \theta)$ 
(the variables $a_{1:n}^{1:N} \in [1,\dots, N]^{n\times N}$ denote the $N$ 
resampled indices at times $i=1,\dots, n$).
Let $\hat p(y_{1:n}|\theta)$ denote the particle filter estimate of the marginal likelihood.
PMCMC is an MCMC method with the target distribution 
$\bbP_n(v_{1:n} | \theta) \hat p(y_{1:n}|\theta) p(\theta)$.}

\blu{Now, we describe a simplified version of the SMC$^2$ algorithm. %final layer of the algorithm.
%Define $z_{1:n}^\ell := (x_{1:n}^\ell,\theta^\ell,x_{1:n}^{\ell-1},\theta^{\ell-1})$,
Let $\bbQ_n(v_{n+1}| v_{1:n},\theta)$
denote the 1-step transition kernel of the particle filter, so that 
$\bbP_n(v_{1:n+1} | \theta) = \bbQ_n(v_{n+1}| v_{1:n},\theta) \bbP_n(v_{1:n} | \theta)$.
%and envision that for a fixed $n$ we have an MCMC kernel 
Suppose $K_n$ is an MCMC %denote a PMCMC 
kernel targeting
$\Pi_n(v_{1:n}, \theta) \propto Q_n(v_{1:n}, \theta)$,
%and such that 
where  one can construct estimates of $p(x_{1:n}, \theta | y_{1:n})$ %can be estimated 
from an appropriate marginal of $\Pi_n$.
%for some auxiliary variables $V \in \cV$, such that
%$\tilde p_n^\ell(z^\ell_{1:n}) = \int_\cV \Pi_n(z_{1:n}^\ell, V) dV$.
%Think of this as the PMCMC kernel.
Define %$u_{1:n} := (z_{1:n}^\ell, V)$ and
\begin{equation}\label{eq:smc2kern}
M_n(v_{1:n},dv'_{1:n+1}) = K_{n}(v_{1:n}, dv'_{1:n}) \otimes \bbQ_n(v_{1:n}',v_{n+1}') dv_{n+1}'\, .
\end{equation}
It is now possible to run an SMC sampler to sequentially target $\Pi_n$.
For $i=1,\dots, N$, one draws $\hat v_1^i := v^i \sim \Pi_0$,
and then iterates for $n \geq 1$
\begin{itemize}
\item Simulate $v_{n+1}^i \sim M_n(\hat v_{1:n}^i, \cdot)$ ;
\item Resample $\hat v_{1:n+1}^i = v_{1:n+1}^j$,
with probability proportional to
$$
\frac{Q_{n+1}(v_{1:n+1}^j)}{Q_{n}(v_{1:n}^{j}) \bbQ_n(v^j_{1:n},v^j_{n+1})} \, .
$$
\end{itemize}
Note that one does not need to be able evaluate either 
$Q_{n}$ or $\bbQ_n$, %(v^j_{1:n},v^j_{n+1})$, 
but only simulate from them %$\bbQ_n$ 
and evaluate the ratio above.
%When the inner kernel is PMCMC, this 
A very similar method, details of which will be given in Sec. \ref{sec:method},
 is referred to as SMC$^2$:
the first (inner) SMC appears in the PMCMC kernel $K_n$, for each $n$,
and the second (outer) SMC appears in the SMC sampler on the extended state-space. 
%At such time that one wants an estimate
%$\bbE_{\tilde p_n^\ell}[f]$, one simply computes the
%estimate using the marginal of $\Pi_n$.
%The cartoon is illustrated in Figure \ref{fig:hmm} (right).
%\begin{figure}
%	\centering
%	\includegraphics[width=.7\columnwidth]{smc2}
%	\caption{{A cartoon of the SMC$^2$ method.
%	The outer SMC sampler recursively updates a target $\Pi_n$
%	on the space of $\theta$ and all auxiliary random variables
%	involved in a particle filter targeting
%	$p(x_{1:n}  | \theta, y_{1:n})$.
%	An appropriate marginal of $\Pi_n$ is $p(x_{1:n} ,\theta | y_{1:n})$.
%	Each mutation of this SMC sampler
%	involves a PMCMC chain with this same target $\Pi_n$.
%	%=\bbP(X_{1:n}|\theta, y_{1:n}) \bbP(\theta| y_{1:n})$.
%	The PMCMC in turn executes a inner SMC (particle filter) targeting
%	$p(x_{1:n} | \theta, y_{1:n})$ at each step.
%	One step of the particle filter extends the state to time $n+1$ and the
%	process repeats.}}
%	\label{fig:smc2}
%\end{figure}
}

\section{Rigorous Problem Formulation}\label{sec:prob_form}

\subsection{Model}

Let $(\mathsf{Y},\mathcal{Y})$ and $(\mathsf{X},\mathcal{X})$ be measurable spaces. {We consider a pair of stochastic processes indexed by a time parameter.}
We are given a sequence of %discrete and regularly in time
observations {$y_0,y_1,\dots$ which are realizations of a discrete-time process $\{Y_n\}_{n\in\mathbb{N}_0}$, $Y_n\in\mathsf{Y}$, where the time between observations
is one unit.}
These observations are associated {with} a continuous-time (Markov) stochastic process $\{X_t\}_{t\geq 0}$, with $X_t\in\mathsf{X}$.
The process would %is
typically arise from the finite-time evolution of %be %associated to
an SPDE,
although we do not make this constraint at this time.

{We now present the stochastic model which describes the probabilistic structure of the processes  $\{Y_n\}_{n\in\mathbb{N}_0}$ and $\{X_t\}_{t\geq 0}$.
In our model, $\theta\in\Theta\subseteq\mathbb{R}^k$ is a static parameter associated to the model. We will define the afore-mentioned structure, conditonal upon $\theta$
and then define a prior probability distribution on this static parameter.
Let $X_{0:n}=(X_0,\dots,X_n)$ correspond to %the discretization  THERE IS NO DISCRETIZATION, IT IS A FINITE SKELETON OF THE PROCESS
a discrete-time skeleton of $\{X_t\}_{t\geq 0}$ on the grid $0:n$.
We are interested in the posterior probability distribution of $(X_{0:n},\theta)$
conditional on observed data $y_0,\dots,y_n$, %recursively in
sequentially over discrete unit times ($n$).}
It is supposed that for any $n\geq 0$, $A\in\mathcal{Y}$
$$
{\mathbb{P}(Y_n\in A | y_{0:n-1}, \{x_t\}_{t\in[0,\dots,n]},\theta) =
\int_A g_{\theta}(x_n,y)dy}
$$
where $dy$ is a $\sigma-$finite measure on $(\mathsf{Y},\mathcal{Y})$ and
for each {%$\theta\in\Theta$
$(\theta,x)\in\Theta\times\mathsf{X}$,
$g_{\theta}(x,\cdot ):\mathsf{Y}\rightarrow\mathbb{R}_+$ is
a probability density on $\mathsf{Y}$.}
{For each $\theta\in\Theta$, $f_{\theta}:\mathsf{X}^2\rightarrow\mathbb{R}_+$, (resp.~$\mu_{\theta}:\mathsf{X}\rightarrow\mathbb{R}_+$)} are the transition density over unit time (resp.~initial density) of $\{X_t\}_{t\geq 0}$ {(resp.~$X_0$)} w.r.t.~a dominating $\sigma-$finite measure on $(\mathsf{X},\mathcal{X})$. Note that for every $(\theta,x)\in\Theta\times\mathsf{X}$,
{$f_{\theta}(x,\cdot): \mathsf{X} \rightarrow \bbR_+$
is a probability density, and for any $n\geq 1$
$$
\bbP(X_{n} \in A | \{x_{r}\}_{0\leq r \leq n-1}, \theta)
= \int_A f_{\theta}(x_{n-1}, x) dx \, .
$$ }
Let $\nu$ be a probability density w.r.t.~Lebesque measure (written $d\theta$) on $(\Theta,\mathcal{B}(\Theta))$ with $\mathcal{B}(\Theta)$ the Borel sets.

{For $n\geq 0$, the posterior probability density on $\mathsf{X}^{n+1}\times\Theta$ that is induced by this construction is given by
\begin{equation}
\label{eq:post_def}
\pi_{n}(x_{0:n},\theta) \propto \nu(\theta)\mu_{\theta}(x_0) g_{\theta}(x_0,y_0)\prod_{p=1}^n f_{\theta}(x_{p-1},x_p)g_{\theta}(x_p,y_p) \, .
\end{equation}
In other words for $A \in \vee^{n+1}\mathcal{X} \vee \mathcal{B}(\Theta)$
$$
\bbP((X_{0:n},\theta) \in A | y_{0:n}) = \int_A \pi_{n}(x_{0:n},\theta)
d(x_{0:n},\theta) \, .
$$}
{Henceforth, we will suppress the dependence on $y_{0:n}$
throughout the article}.
Let $\varphi:\mathsf{X}^{n+1}\times\Theta\rightarrow\mathbb{R}$ be integrable w.r.t.~$\pi_n$.  % , then o
{Our} objective is to compute, recursively in $n$
$$
\mathbb{E}_{\pi_{n}}[\varphi(X_{0:n},\theta)]:=\int_{\mathsf{X}^{n+1}\times\Theta} \varphi(x_{0:n},\theta)\pi_{n}(x_{0:n},\theta)d(x_{0:n},\theta)
$$
{where we use the notation
$\mathbb{E}_{\pi}$ to denote expectations w.r.t.~
a probability density/measure $\pi$.
The role of the function $\varphi$ is as a summary or quantity of interest, relating to the random variables $X_{0:n},\theta$. That is, our objective is to compute expectations, such as moments,
w.r.t.~the posterior, with density defined in \eqref{eq:post_def}.
We note that
%even if $f_{\theta}(x,x'),\mu_{\theta}(x),g_{\theta}(x,y)$ are available (in the sense that they
%can be simulated from or estimated unbiasedly) then
one often must use %advanced
Monte Carlo methods to approximate the sequence of expectations.

Before concluding this section we mention a canonical statistical model,
the conditionally Gaussian model.

{
\begin{exam}\label{ex:gauss}
Let $N(m,C)$ denote a (possibly infinite-dimensional)
Gaussian random variable with mean $m$ and covariance operator $C$, and
let $\phi(\cdot ; m, C)$ denote its density (with respect to some dominating measure,
which may be taken as Lebesgue in finite dimensions).
Assume $X_0 \sim N(m_0, \Sigma_0)$. For each $\theta \in \Theta$,
let $\Psi_\theta: \mathsf{X} \rightarrow \mathsf{X}$ and
$h_\theta: \mathsf{X} \rightarrow \mathsf{Y}$ be continuous
and let $\Sigma_\theta, \Gamma_\theta$ be symmetric positive definite
operators. An example of a model is, for $n\geq 0$
\begin{eqnarray}\label{eq:gaussian}
X_{n+1}|X_n & \sim N(\Psi_\theta(X_n), \Sigma_\theta) \, , \\
\nonumber
Y_{n}|X_n & \sim N(h_\theta(X_{n}), \Gamma_\theta) \, .
\end{eqnarray}
This model is ubiquitous in the data assimilation literature \cite{law2015data}.
Once $\nu$ is specified, the model \eqref{eq:post_def} is given by
%\begin{eqnarray}
$\mu_\theta(x_0)  = \phi(x_0 ; m_0, \Sigma_0) , ~%\\
f_\theta(x,x') = \phi(x' ; \Psi_\theta(x), \Sigma_\theta) , ~ {\rm and}~
g_\theta(x,y)  = \phi(y ; h_\theta(x), \Gamma_\theta)$.

This model fits into the context of this paper when $\mathsf{X}$
is infinite dimensional, e.g.~a Hilbert space, and $\Psi_\theta$
is the solution of a PDE parametrized by $\theta$.
\end{exam}}}
%\end{eqnarray}}
%\example}

\subsection{Discretized Model}\label{sec:disc_model}

%In practice
{The exposition here closely follows that developed in \cite{jklz}.}
Here we explicitly assume
%that $f_{\theta}(x,x'),\mu_{\theta}(x),g_{\theta}(x,y)$ cannot be simulated from or estimated unbiasedly.
%This can occur for instance if $\mathsf{X}$ is an infinite dimensional space. {In this scenario, advanced Monte Carlo methods cannot
%typically not be used.} Therefore, we assume that
one must work
with a discretized version of the model, that is, there does not (currently) exist an unbiased and non-negative approximation of $\pi_{n}(x_{0:n},\theta)$.
{We remark that if the latter approximations are available, then the strategy to be outlined is not required.}

{Set $\alpha\in\mathbb{N}_0^d$, which will refer to a collection of indices which will denote
the level of discretization of our model {in each of $d$ dimensions}.
That is, as the components of $\alpha$ increase, so does the accuracy of the approximation. Precise examples are given in Section \ref{sec:numerics}.}
More explicitly, {for any fixed $\alpha\in\mathbb{N}_0^d$, let $(\mathsf{X}_{\alpha},\mathcal{X}_{\alpha})$ and $(\mathsf{Y}_{\alpha},\mathcal{Y}_{\alpha})$ be measurable spaces such that for every $n\geq 0$}
one  can obtain a biased approximation $X_{\alpha}\in\mathsf{X}_{\alpha}\subseteq\mathsf{X}$
 of $X_n$, and $Y_{\alpha}\in\mathsf{Y}_{\alpha}\subseteq\mathsf{Y}$ of $Y_n$. That is, one can define the probability density for $n\geq 0$, on $\mathsf{X}_{\alpha}^{n+1}\times\Theta$:{
\begin{equation}\label{eq:post_disc}
\pi_{n,\alpha}(x_{0:n},\theta) \propto \nu(\theta)\mu_{\theta,\alpha}(x_0) g_{\theta,\alpha}(x_0,y_0)\prod_{p=1}^n f_{\theta,\alpha}(x_{p-1},x_p)g_{\theta,\alpha}(x_p,y_p)
\end{equation}}
where $y_n\in\mathsf{Y}_{\alpha}$ for each $n\geq 0$.
Here for every {$(\alpha,\theta)\in\mathbb{N}_0^d\times\Theta$
\begin{itemize}
\item{For all $x\in\mathsf{X}_{\alpha}$,
$g_{\theta,\alpha}(x,\cdot)$ is a probability density on $\mathsf{Y}_\alpha$ ;}
\item{For all $x\in\mathsf{X}_{\alpha}$, $f_{\theta,\alpha}(x,\cdot)$ is a probability density on $\mathsf{X}_\alpha$; }
\item{$\mu_{\theta,\alpha}$ is a probability density on $\mathsf{X}_\alpha$.}
\end{itemize}
}

{
Consider $\varphi:\mathbb{N}_0^d\times\mathsf{X}^{n+1}\times\Theta\rightarrow\mathbb{R}$, where %we make the defintion that
for any $(x_{0:n},\theta)\in \mathsf{X}^{n+1}\times\Theta$
$$
\lim_{\min_{1\leq i\leq d}\alpha_i \rightarrow+\infty}\varphi_{\alpha}(x_{0:n},\theta) = \varphi(x_{0:n},\theta).
$$
It is assumed that we have
\begin{eqnarray}
\mathbb{E}_{\pi_{n,\alpha}}[\varphi_{\alpha}(X_{0:n},\theta)] & \neq & \mathbb{E}_{\pi_n}[\varphi(X_{0:n},\theta)]. \nonumber \\
\lim_{\min_{1\leq i\leq d}\alpha_i \rightarrow+\infty}|\mathbb{E}_{\pi_{n,\alpha}}[\varphi_{\alpha}(X_{0:n},\theta)]-\mathbb{E}_{\pi_n}[\varphi(X_{0:n},\theta)]| & = & 0.\label{eq:assump1}
\end{eqnarray}}
{As remarked in the introduction}, the computational cost associated with $X_{\alpha},Y_{\alpha}$ ({sampling, or evaluating the densities $g_{\theta,\alpha},f_{\theta,\alpha}, \mu_{\theta,\alpha}$}) increases as any index %the values
of $\alpha$ increases.
Our objective is now to compute $\mathbb{E}_{\pi_{n,\alpha}}[\varphi_{\alpha}(X_{0:n},\theta)]$ recursively for each $n \geq 0$.

\subsection{Multi-Index Methods}\label{sec:mimc}

{The approach to be described, provides an approach to approximating $\mathbb{E}_{\pi_{n,\alpha}}[\varphi_{\alpha}(X_{0:n},\theta)]$ for any fixed $\alpha\in\mathbb{N}_0^d$.}
Define the difference operator $\Delta_i$, $i\in\{1,\dots,d\}$ as
$$
\Delta_i \mathbb{E}_{\pi_{n,\alpha}}[\varphi_{\alpha}(X_{0:n},\theta)]
:= \left\{\begin{array}{ll}
\mathbb{E}_{\pi_{n,\alpha}}[\varphi_{\alpha}(X_{0:n},\theta)]- \mathbb{E}_{\pi_{n,\alpha-e_i}}[\varphi_{\alpha-e_i}(X_{0:n},\theta)]  & \textrm{if}~\alpha_i>0 \\
\mathbb{E}_{\pi_{n,\alpha}}[\varphi_{\alpha}(X_{0:n},\theta)] & \textrm{otherwise}
\end{array}\right.
$$
where $e_i$ are the canonical vectors on $\mathbb{R}^d$. {Set
\begin{equation}\label{eq:multiincrement}
\Delta\mathbb{E}_{\pi_{n,\alpha}}[\varphi_{\alpha}(X_{0:n},\theta)] := (\Delta_1\circ\Delta_2\circ\cdots\circ \Delta_d)\Big(\mathbb{E}_{\pi_{n,\alpha}}[\varphi_{\alpha}(X_{0:n},\theta)]\Big)
\end{equation}
%that is, $\Delta$ is a 
where $(\Delta_1\circ\Delta_2\circ\cdots\circ \Delta_d)$ denotes 
the composition of $\Delta_1,\dots,\Delta_d$.
Note that the order of applying the operators $\Delta_i$ in $\Delta$ does not matter.}

We now consider the identity
$$
\mathbb{E}_{\pi_{n}}[\varphi(X_{0:n},\theta)] = 
\sum_{\alpha\in \mathbb{N}_0^d} \Delta \mathbb{E}_{\pi_{n,\alpha}}[\varphi_{\alpha}(X_{0:n},\theta)].
$$
The work \cite{mimc} proposes to leverage this identity by constructing a
%propose the
biased estimator of $\mathbb{E}_{\pi_{n}}[\varphi(X_{0:n},\theta)]$
for some finite multi-index set $\cI \subset \bbN_0^d$ as follows
%for $\mathcal{I}_{ m_1:m_d}  := \{\alpha\in\mathbb{N}_0^d:\alpha_1\in\{0,\dots,m_1\},\dots,\alpha_d\in\{0,\dots,m_d\}\}$
{
\begin{equation}\label{eq:mimc}
\mathbb{E}_{\pi_{n} ,\mathcal{I}}[\varphi(X_{0:n},\theta)] :=
\sum_{\alpha\in\mathcal{I}}\Delta \mathbb{E}_{\pi_{n,\alpha}}[\varphi_{\alpha}(X_{0:n},\theta)] \, .
\end{equation}}
%where $\mathcal{I}\subset\mathbb{N}_0^d$ such that $\alpha\in\mathcal{I}$ and for any $\alpha'\neq\alpha$, $\alpha'\in\mathcal{I}$, $\sum_{j=1}^d\alpha_j'<\sum_{j=1}^d\alpha_j$.
%These constraints mean that $\mathbb{E}_{\pi_{n,\alpha}}[\varphi_{\alpha}(X_{0:n},\theta)]=\mathbb{E}_{\pi_{n} ,\mathcal{I}}[\varphi(X_{0:n},\theta)]$.}
%\eqref{eq:mimc}
Such an approximation strategy is inspired by work in 
the sparse grids literature \cite{bungartz2004sparse}.
This estimator can (in principle) be approximated by a Monte Carlo method.
{This can be achieved by (if possible) sampling from a coupling of the (at most) $2^d$ different probability measures for a given $\alpha \in \mathcal{I}$ ({see \cite{mimc} for details}). It is counterintuitive at first
to construct a single estimator from a sum of so many other estimators,
but in fact if the coupling is strong, and under appropriate assumptions,
then this can substantially more efficient than a single term estimator,
in the sense that smaller MSE can be achieved for the same cost.
It is remarked that sampling from such a coupling is very challenging, especially in the context of the model considered in Section \ref{sec:disc_model}.}
The residual error is given by
$$%\begin{equation}\label{eq:mimcerror}
\mathbb{E}_{\pi_{n}}[\varphi(X_{0:n},\theta)]-\mathbb{E}_{\pi_{n},\mathcal{I}}[\varphi(X_{0:n},\theta)] =
\sum_{\alpha \notin \mathcal{I}} \Delta \mathbb{E}_{\pi_{n,\alpha}}[\varphi_{\alpha}(X_{0:n},\theta)].
$$%\end{equation}
{It is shown in \cite{mimc, jklz1} that under appropriate assumptions
on the convergence of
estimates of the individual terms in \eqref{eq:multiincrement} one can gain
significant `speed-up' relative to single term or even single index MLMC methods.
The key point we emphasize here is that \emph{all results of \cite{mimc},
pertaining to all different index sets} $\mathcal{I}$,
rely solely on the convergence properties of estimates of the individual terms \eqref{eq:multiincrement}. Since there are significant other difficulties to deal with in the
present work, the finer properties of estimators with various different index sets
$\mathcal{I}$ will not be considered here,
although we note this is a crucial consideration in practice. Our objective here will
rather be to establish a general proof of principle method which provides convergence of
estimates of the individual terms in \eqref{eq:multiincrement} under suitable assumptions.
The results will be illustrated in Section \ref{sec:numerics}.}

\subsection{Renormalized Multi-Index Identity}

The following idea builds upon %is a combination of
the approaches in \cite{jklz} and \cite{jklz1}.
%We c
Consider \eqref{eq:mimc} and {in particular
consider a single given summand \eqref{eq:multiincrement}} for $\alpha\in\mathcal{I}$, with $n$ fixed. %We s
%Suppose that
{This summand is itself a linear combination of expectations with respect to}
$1<k_{\alpha}\leq 2^d$ probability measures. % for which one wants to compute an .
These $k_{\alpha}$ probability measures induce $k_{\alpha}'$ differences in \eqref{eq:mimc}; %. E.g.~
if $k_{\alpha}=2^d$, then $k_{\alpha}'=2^{d-1}$. {We remark that in the case that $k_{\alpha}=1$, one  does not need to consider how to construct a coupling for an MIMC method as the summand \eqref{eq:multiincrement} is only an expectation w.r.t.~a single probability measure.}
%of \eqref{eq:mimc}.}

For simplicity of notation we will denote %write
the $k_{\alpha}$ multi-indices by $\alpha(1),\dots,\alpha(k_{\alpha})$, where for $i\in\{1,\dots,k_{\alpha}\}$,
$\alpha(i)\in \mathcal{I}$. {Let $\alpha(i)_j$ denote the $j^{\textrm{th}}-$element of $\alpha(i)$.}
{The convention of the (non-unique) labelling is such that,
$\sum_{j=1}^d[\alpha(2i)-\alpha(2i-1)]_j=1$ for each $i\in\{1,\dots,k_{\alpha}'\}$,
%$\sum_{j=1}^d[\alpha(i)-\alpha(i-1)]_j\geq {-1}${
%(AJAY: is it for sure possible to do this (with $-1$) for any value of $d$?  This is not unique, right? Should we just require the first condition to hold, and allow the ordering of the pairs to be arbitrary?)}
%for each $i\in\{2,\dots,k_{\alpha}\}$
$\alpha(k_{\alpha})=\alpha$ and $\alpha(1)=(\max\{\alpha_1-1,0\},\dots,\max\{\alpha_d-1,0\})$.} {%The labelling of the indices associated to
%$\alpha(1:k_{\alpha})=(\alpha(1),\dots,\alpha(k_{\alpha}))$, is such that the cost associated to approximating $\pi_{n,\alpha(i)}$ is not smaller than that of $\pi_{n,\alpha(j)}$,
%when $i>j$, $(i,j)\in\{1,\dots,k_{\alpha}\}^2$.
This labelling will provide a convenient way to write
$\Delta \mathbb{E}_{\pi_{n,\alpha}}$ $[\varphi_{\alpha}(X_{0:n},\theta)]$ below.
\begin{exam}
Suppose $d=3$ and $\alpha=(2,2,2)$, so $k_{\alpha}=8$, $k_{\alpha}'=4$. Then a labelling which satisfies the above constraints is
$$
\alpha(1) = (1,1,1), \alpha(2) = (2,1,1), \alpha(3) = (1,1,2), \alpha(4) = (2,1,2),
$$
$$
\alpha(5) = (1,2,1), \alpha(6) = (2,2,1), \alpha(7) = (1,2,2), \alpha(8) = (2,2,2).
$$
\end{exam}
}

%\mathsf{X}_{\alpha(1)}\times\cdots\times\mathsf{X}_{\alpha(i-1)}\times A_i\times\mathsf{X}_{\alpha(i+1)}\times\cdots\times\mathsf{X}_{\alpha(k_{\alpha})}

{Recall the form of the target \eqref{eq:post_disc}
and the multi-increment summand to be estimated \eqref{eq:multiincrement}. }
We suppose that there exists a coupling of the discretized dynamics.
That is, there exists a Markov density $\check{f}_{\theta,\alpha(1:k_{\alpha})}(x(1:k_{\alpha}),x'(1:k_{\alpha}))$
such that for any $x(1:k_{\alpha})={(x(1),\dots,x(k_{\alpha}))}\in\bigotimes_{i=1}^{k_{\alpha}}\mathsf{X}_{\alpha(i)}$ and any $i\in\{1,\dots,k_{\alpha}\}$,
$A_i\in\mathcal{X}_{\alpha(i)}$, we have:
$$
\int_{\bigotimes_{j=1}^{i-1} \mathsf{X}_{\alpha(j)} \times A_i\times \bigotimes_{j=i+1}^{k_{\alpha}} \mathsf{X}_{\alpha(j)}}\check{f}_{\theta,\alpha(1:k_{\alpha})}(x(1:k_{\alpha}),x'(1:k_{\alpha}))dx'(1:k_{\alpha})
=
$$
\begin{equation}\label{eq:coupleddyn}
\int_{A_i}f_{\theta,\alpha(i)}(x(i),x'(i))dx'(i).
\end{equation}
{In other words, for a given $\alpha$,
coupled Markov dynamics are performed on the hierarchy of
$k_\alpha$ meshes in such a way that the marginal of the coupled dynamics
with respect to any of the $k_\alpha$ meshes $\alpha(i)$
corresponds to the exact Markov dynamics on mesh $\alpha(i)$.
Importantly, we do not assume that we can evaluate this Markov density.
We will only require being able to simulate from it.}

Similarly, we suppose that there exists a probability density
$\check{\mu}_{\theta}$ on $\bigotimes_{i=1}^{k_{\alpha}}\mathsf{X}_{\alpha(i)}$
such that for any $i\in\{1,\dots,k_{\alpha}\}$,
$A_i\in\mathcal{X}_{\alpha(i)}$, we have:
$$
\int_{\bigotimes_{j=1}^{i-1} \mathsf{X}_{\alpha(j)} \times A_i\times \bigotimes_{j=i+1}^{k_{\alpha}} \mathsf{X}_{\alpha(j)}}\check{\mu}_{\theta,\alpha(1:k_{\alpha})}(x(1:k_{\alpha}))dx(1:k_{\alpha})
=
$$
\begin{equation}\label{eq:coupledprior}
\int_{A_i}\mu_{\theta,\alpha(i)}(x(i))dx(i).
\end{equation}
We remark that one can find scenarios for which this is true.
We give an example below and another %specific example
later in Section \ref{sec:numerics}.

{\begin{exam}
As a concrete example,
consider the setting of \eqref{eq:gaussian} in Example \ref{ex:gauss},
where $\Psi_\theta$ is the forward solution of a PDE over a unit time interval,
for example Navier-Stokes equation as in \cite{law2012evaluating}.
Suppose $\mu_{\theta} = N(0,S^a)$, where $a>0$ and
$S$ is the (possibly weak) solution operator
$S: f \mapsto u$
of an elliptic PDE on a cubic domain $\Omega = [0,1]^d$,
with convex boundary $\partial \Omega$:
\[\begin{split}
\left (-\sum_{i=1}^d \partial^2u / \partial x_i^2
\right )(x) & = f(x) \, , \qquad x \in \Omega \, , \\
u(x) & =  0 \, , \qquad x \in \partial \Omega \, .
\end{split}
\]
Suppose we have eigenfunctions $\{v_k\}_{k\in \bbZ_+^d}$ such that
$S v_k = \lambda_k v_k$ for $k\in \bbZ_+^d$,
and then we have a spectral (Karhunen-Lo{\'e}ve)
expansion of $X_{0} \sim \mu_{\theta}$ as
$X_0 = \sum_{k \in \bbZ_+^d} \lambda_k^{a/2} \xi_k v_k$,
where $\xi_k \sim N(0,1)$ i.i.d. (see e.g. \cite{stuart}).
Note the eigenfunctions can be decomposed as products of eigenfunctions
$\{\psi_l\}_{l \in \bbZ_+}$ of the $d=1$ problem
$(\partial^2\psi_{l}/\partial x_1^2)(x_1) = \lambda_{l} \psi_l(x_1)$, i.e.
$v_k(x_1,\dots, x_d) = \prod_{i=1}^d \psi_{k_i}(x_i)$.
Furthermore, suppose that
$X_{0,\alpha} = \sum_{k \in \mathcal{T}_\alpha} \lambda_k^{a/2} \xi_k v_k$,
where $\mathcal{T}_\alpha = \{k\in \bbZ_+^d ;
1\leq k_1 \leq 2^{\alpha_1}, \cdots, 1\leq k_d \leq 2^{\alpha_d} \}$.
%is the index set corresponding to ,
%and $\psi_{j}$ is the $j^{\rm th}$ single variable eigenfunction
%(i.e. corresponding to $d=1$).
One can simulate
$$
X_{0,\alpha(k_\alpha)} =
\sum_{k \in \mathcal{T}_{\alpha(k_\alpha)}} \lambda_k^{a/2} \xi_k v_k
\sim \mu_{\theta,\alpha(k_\alpha)} \, ,
$$
and then coarsen this realization appropriately such that
$X_{0,\alpha(i)} =
\sum_{k \in \mathcal{T}_{\alpha(i)}} \lambda_k^{a/2} \xi_k v_k$
is a realization of $\mu_{\theta,\alpha(i)}$, for each of the other targets $i<k_\alpha$.
Note that $\mathcal{T}_{\alpha(i)} \subset \mathcal{T}_{\alpha(k_\alpha)}$
so this just consists in using the same
$\{\xi_k\}_{k \in \mathcal{T}_{\alpha(k_\alpha)}}$
and setting $\xi_k=0$ for all
$k \in \mathcal{T}_{\alpha(k_\alpha)} \backslash \mathcal{T}_{\alpha(i)}$.
Hence we have a strong coupling which satisfies \eqref{eq:coupledprior}.

For simulating the forward kernel $\Psi_{\theta,\alpha}$,
assume that we have a Galerkin spectral solver \cite{hesthaven}
for any given spectral truncation
level $\alpha$ and a given time-step size
(if time-discretization is considered in the approximation,
its discretization level will be given by index $\alpha_{d+1}$).
Assume $\Sigma_{\theta} = S^b$, for $b>0$,
with $S$ defined as above, and discretizations defined similarly.
One can then simulate a single realization
$\chi_{\alpha(k_\alpha)} \sim N(0, \Sigma_{\theta,\alpha(k_\alpha)})$
just as for the initial condition and coarsen this to
$\chi_{\alpha(i)}$ for driving each of the other dynamics
with $i<k_\alpha$.
Hence we have a strong coupling which satisfies \eqref{eq:coupleddyn}.
\end{exam}}

\blu{Based on \eqref{eq:coupleddyn} and \eqref{eq:coupledprior}
we have an exact $k_\alpha$-fold coupling of
$$
\nu(\theta) \check{\mu}_{\theta,\alpha(1:k_{\alpha})}(x_0(1:k_{\alpha}))  
\prod_{p=1}^n \check{f}_{\theta,\alpha(1:k_{\alpha})}(x_{p-1}(1:k_{\alpha}),x_{p}(1:k_{\alpha})) \, ,
$$
which takes the role of the 2-fold coupling $\check F_{l,l-1}$ given in \eqref{eq:checkf}. 
%in Sec. \ref{sec:cartoon}.
Let $\check{g}:{\mathbb{N}_0^{d\times k_\alpha}}\times\bigotimes_{i=1}^{k_{\alpha}}\mathsf{X}_{\alpha(i)}\times\Theta\times \mathsf{Y} \rightarrow (0,\infty)$
be %(at stage)
arbitrary for the moment. 
%Letting $\prod_{p=0}^n\check{g}_{\theta,\alpha(1:k_{\alpha})}(x_p(1:k_{\alpha}),y_p)$ take the role of $J(x)$
We consider the
{following probability density on the space $(\bigotimes_{i=1}^{k_{\alpha}}\mathsf{X}_{\alpha(i)})^{n+1}\times \Theta$}
\begin{eqnarray}\label{eq:jointmimc}
\xi_{n,\alpha(1:k_{\alpha})}(x_{0:n}(1:k_{\alpha}),\theta) & \propto & \nu(\theta) \check{\mu}_{\theta,\alpha(1:k_{\alpha})}(x_0(1:k_{\alpha}))  \prod_{p=1}^n \check{f}_{\theta,\alpha(1:k_{\alpha})}(x_{p-1}(1:k_{\alpha}),x_{p}(1:k_{\alpha})) \times 
\\ \nonumber & &
 \prod_{p=0}^n \check{g}_{\theta,\alpha(1:k_{\alpha})}(x_p(1:k_{\alpha}),y_p) \, .
\end{eqnarray}}
In the works \cite{jklz,jklz1} the following choice is made, for $p=0,\dots,n$,
\begin{equation}\label{eq:g}
\check{g}_{\theta,\alpha(1:k_{\alpha})}(x_p(1:k_{\alpha}),y_p) = \max\{g_{\theta,\alpha(1)}(x_p(1),y_p),\dots,g_{\theta,\alpha(k_{\alpha})}(x_p(k_{\alpha}),y_p)\} \, ,
\end{equation}
and this is the choice used in this article.
\blu{The motivation will become apparent shortly.
It is noted that other choices are possible 
(see e.g. \cite{franks} for further discussion and some alternatives).}
%{The important property is that there is some $C < +\infty$
%such that for all $i=1,\dots, k_\alpha$
%$g_{\theta,\alpha(i)}(x_p(i),y_p) /\check{g}_{\theta,\alpha(1:k_{\alpha})}(x_p(1:k_{\alpha}),y_p) \leq C$. } **Well $C=1$ and that's obvious (and not correctly written)**

\blu{The expression {\eqref{eq:multiincrement} will be approximated 
using samples distributed according to \eqref{eq:jointmimc}.}} 
{We start by considering a single expectation in \eqref{eq:multiincrement}.
Note that \eqref{eq:coupleddyn} and \eqref{eq:coupledprior},
and the form of $\xi_{n, \alpha(1:k_\alpha)}$, immediately imply that }
%Now we have, {under our assumptions}, that
for any $\alpha(i)$, $i\in\{1,\dots,k_{\alpha}\}$
%$$
%\mathbb{E}_{\pi_{n,\alpha(i)}}[\varphi_{\alpha(i)}(X_{0:n},\theta)] =
%$$
\begin{equation}\label{eq:ref_2}
\mathbb{E}_{\pi_{n,\alpha(i)}}[\varphi_{\alpha(i)}(X_{0:n},\theta)] = \frac{\mathbb{E}_{\xi_{n,\alpha(1:k_{\alpha})}}\big[\varphi_{\alpha(i)}(X_{0:n}(i),\theta)
\prod_{p=0}^n \frac{g_{\theta,\alpha(i)}(X_p(i),y_p)}{\check{g}_{\theta,\alpha(1:k_{\alpha})}(X_p(1:k_{\alpha}),y_p)}
\big]}%\Big/
{\mathbb{E}_{\xi_{n,\alpha(1:k_{\alpha})}}\big[\prod_{p=0}^n \frac{g_{\theta,\alpha(i)}(X_p(i),y_p)}{\check{g}_{\theta,\alpha(1:k_{\alpha})}(X_p(1:k_{\alpha}),y_p)}\big]} \, .
\end{equation}
\blu{This is recognizable as the analogue of \eqref{eq:marginal}.}
%For ease of notation, we set 
The following notation will be used for the weights, 
for any $\alpha(i)$, $i\in\{1,\dots,k_{\alpha}\}$
$$
H_{i,n,\alpha,\theta}(x_{0:n}(1:k_{\alpha})) := \prod_{p=0}^n \frac{g_{\theta,\alpha(i)}(x_p(i),y_p)}{\check{g}_{\theta,\alpha(1:k_{\alpha})}(x_p(1:k_{\alpha}),y_p)} \, .
$$
\blu{Note that the form of  \eqref{eq:g} ensures the weights are bounded (by 1), 
which is desirable for stability of the algorithm.}
{By combining \eqref{eq:ref_2} with $H_{i,n,\alpha,\theta}(x_{0:n}(1:k_{\alpha}))$ and recalling the definitions of $\Delta$ and $\Delta_i$ given in and above \eqref{eq:multiincrement},
one can then deduce that}
%$$
\begin{eqnarray}\nonumber
\Delta \mathbb{E}_{\pi_{n,\alpha}}[\varphi_{\alpha}(X_{0:n},\theta)]  & = &
 %$$
%$$
%\Delta \mathbb{E}_{\pi_{n,\alpha}}[\varphi_{\alpha}(X_{0:n},\theta)]  & =  &
\sum_{i=1}^{k_{\alpha}'}\tau_{i,\alpha}
\Bigg(\frac{
\mathbb{E}_{\xi_{n,\alpha(1:k_{\alpha})}}[\varphi_{\alpha(2i)}(X_{0:n}(2i),\theta)H_{2i,n,\alpha,\theta}(X_{0:n}(1:k_{\alpha}))]
}{\mathbb{E}_{\xi_{n,\alpha(1:k_{\alpha})}}[H_{2i,n,\alpha,\theta}(X_{0:n}(1:k_{\alpha}))]} \\ 
&  & - %\nonumber\\ & &
%$$
%\begin{equation}
\frac{\mathbb{E}_{\xi_{n,\alpha(1:k_{\alpha})}}[\varphi_{\alpha(2i-1)}(X_{0:n}(2i-1),\theta)H_{2i-1,n,\alpha,\theta}(X_{0:n}(1:k_{\alpha}))]}
{\mathbb{E}_{\xi_{n,\alpha(1:k_{\alpha})}}[H_{2i-1,n,\alpha,\theta}(X_{0:n}(1:k_{\alpha}))}
\Bigg) \label{eq:basic_idea}
\end{eqnarray}
where $|\alpha|=\sum_{j=1}^d\alpha_j$ and {$\tau_{i,\alpha}=(-1)^{|\alpha(k_{\alpha})-\alpha(2i)|}$}.
\blu{This is the multi-index analogue of \eqref{eq:increment}, and has been used before in \cite{jklz}.
This approach is called approximate coupling; see \cite{ml_rev} for further discussion.}

Our strategy {for approximating $\mathbb{E}_{\pi_{n},\mathcal{I}}[\varphi(X_{0:n},\theta)]$}
is then the following. {Noting
\eqref{eq:mimc} and \eqref{eq:basic_idea}, we will approximate each summand in \eqref{eq:mimc}, by approximating the r.h.s.~of  \eqref{eq:basic_idea}. This will be achieved as follows.}
Independently for each $\alpha\in\mathcal{I}$ (with $k_{\alpha}>1$), and serially for each $n$,
 we will sample (approximately) from $\xi_{n,\alpha(1:k_{\alpha})}(x_{0:n}(1:k_{\alpha}),\theta)$, and compute a Monte Carlo estimate of $\Delta \mathbb{E}_{\pi_{n,\alpha}}[\varphi_{\alpha}(X_{0:n},\theta)] $. As noted above, in the case $k_{\alpha}=1$, one can simply sample from $\pi_{n,\alpha}$ as no coupling is required.

\section{Simulation Strategy}\label{sec:method}

{The purpose of this Section is to describe a method to approximate $\mathbb{E}_{\pi_{n},\mathcal{I}}[\varphi(X_{0:n},\theta)]$. This will be achieved
by using the SMC$^{2}$ method to approximate expectations w.r.t.~$\xi_{n,\alpha(1:k_{\alpha})}$
for each $\alpha\in\mathcal{I}$ (with $k_{\alpha}>1$), recursively in $n$.
If $k_{\alpha}=1$, then one can simply use the standard SMC$^{2}$
method %as previously stated, except
considering $\pi_{n,\alpha}$.
The SMC$^{2}$ approach uses the particle MCMC method \cite{andrieu}, which in turn relies upon particle filters. Therefore, to develop our approach, we first provide a review of particle filters, followed by particle MCMC in the present context.}

\subsection{Particle Filter}

In this section, we focus upon the approximation of the
{density of the state conditional on fixed parameter $\theta$, given by %, recursively in $n$
%\begin{eqnarray*}
 $\xi_{n,\alpha(1:k_{\alpha}),\theta}(x_{0:n}(1:k_{\alpha})) \propto
\xi_{n,\alpha(1:k_{\alpha})}(x_{0:n}(1:k_{\alpha}),\theta)$.}
%& \propto & \check{\mu}_{\theta,\alpha(1:k_{\alpha})}(x_0(1:k_{\alpha})) \check{g}_{\theta,\alpha(1:k_{\alpha})}(x_0(1:k_{\alpha}),y_0) \times \\ & &
%\prod_{p=1}^n \check{f}_{\theta,\alpha(1:k_{\alpha})}(x_{p-1}(1:k_{\alpha}),x_{p}(1:k_{\alpha})) \check{g}_{\theta,\alpha(1:k_{\alpha})}(x_p(1:k_{\alpha}),y_p) \, .
%\end{eqnarray*}
%That is, $\theta$ is fixed.
It is natural to adopt an SMC approach, in which we sequentially perform importance sampling
and then resampling.
This procedure is the standard particle filter which is given in Algorithm \ref{algo:pf}.
\blu{The number of particles $N$ will be fixed once and for all, but it is noted that 
this is an important consideration for the efficiency of the proposed method.}

\begin{algorithm}[!b]
\caption{The Particle Filter}
\label{algo:pf}
%\begin{algorithmic}
\begin{itemize}
\item \textbf{Initialize:} Set $p=0$, for $i\in\{1,\dots,N\}$ sample $x_0^i(1:k_{\alpha})$ from $\check{\mu}_{\theta,\alpha(1:k_{\alpha})}$ and evaluate the weight
$$
w_{p,\alpha,\theta}^i = \Big(\check{g}_{\theta,\alpha(1:k_{\alpha})}(x_0^i(1:k_{\alpha}),y_0)\Big)\Big(\sum_{j=1}^N \check{g}_{\theta,\alpha(1:k_{\alpha})}(x_0^j(1:k_{\alpha}),y_0) \Big)^{-1}
$$

\item \textbf{Iterate:} Set $p=p+1$,
\begin{itemize}
\item Sample %: resample
$(a_{p-1}^1,\dots,a_{p-1}^N)\in\{1,\dots,N\}^N$, where, independently for each $i\in\{1,\dots,N\}$, $\mathbb{P}(a_{p-1}^i=j)=w_{p-1,\alpha,\theta}^j$.
\item Sample $x_p^i(1:k_{\alpha})|x_{p-1}^{a_{p-1}^i}(1:k_{\alpha})$ from $\check{f}_{\theta,\alpha(1:k_{\alpha})}(x_{p-1}^{a_{p-1}^i}(1:k_{\alpha}),\cdot)$, for $i\in\{1,\dots,N\}$,
and evaluate the weight
$$
w_{p,\alpha,\theta}^i = \Big(\check{g}_{\theta,\alpha(1:k_{\alpha})}(x_p^i(1:k_{\alpha}),y_p)\Big)\Big(\sum_{j=1}^N \check{g}_{\theta,\alpha(1:k_{\alpha})}(x_p^j(1:k_{\alpha}),y_p) \Big)^{-1}.
$$
\end{itemize}
\end{itemize}
%\end{algorithmic}
\end{algorithm}

The joint {density} of all the variables sampled in Algorithm \ref{algo:pf}, up to time $n$, is written
\begin{equation}\label{eq:pfjoint}
\psi_{\alpha,\theta}(x_{0:n}^{1:N}(1:k_{\alpha}),a_{0:n-1}^{1:N})
\end{equation}
{where, for $0\leq k \leq n$,  $x_k^{1:N}(1:k_{\alpha})=(x_k^1(1:k_{\alpha}),\dots,x_k^N(1:k_{\alpha}))\in(\bigotimes_{i=1}^{k_{\alpha}}\mathsf{X}_{\alpha(i)})^N$,
$a_{k}^{1:N}=(a_k^1,\dots,a_k^N)\in\{1,\dots,N\}^N$,
$x_{0:n}^{1:N}(1:k_{\alpha})=(x_0^{1:N}(1:k_{\alpha}),\dots,x_n^{1:N}(1:k_{\alpha}))\in (\bigotimes_{i=1}^{k_{\alpha}}\mathsf{X}_{\alpha(i)})^{(n+1)N}$,
$a_{0:n-1}^{1:N}=(a_{0}^{1:N},\dots,a_{n-1}^{1:N})\in\{1,\dots,N\}^{nN}$.
%That is,
In particular,} $a_{n-1}^i$ is the index at time $n-1$ of the resampled particle which has the index $i$ at time $n$.
%Hence it can be used to trace the ancestral lineage of the particle as it evolves through the sequential importance resampling steps.
%In general, to estimate $\int_{\bigotimes_{i=1}^{k_{\alpha}}\mathsf{X}_{\alpha(i)}^{n+1}} \varphi(x_{n}(1:k_{\alpha}))\xi_{n,\alpha(1:k_{\alpha}),\theta}(x_{0:n}(1:k_{\alpha}))dx_{0:n}(1:k_{\alpha})$ for
%integrable and real-valued functions $\varphi$ one has the consistent estimate for any $n\geq 0$
%$$
%\sum_{i=1}^Nw_{n,\alpha,\theta}^i\varphi(x_{n}^i(1:k_{\alpha})).
%$$
%Although it does not work well, to estimate $\int_{\bigotimes_{i=1}^{k_{\alpha}}\mathsf{X}_{\alpha(i)}^{n+1}} \varphi(x_{0:n}(1:k_{\alpha}))\xi_{n,\alpha(1:k_{\alpha}),\theta}(x_{0:n}(1:k_{\alpha}))dx_{0:n}(1:k_{\alpha})$ for
%integrable and real-valued functions $\varphi$, one can do the following.
For each $i\in\{1,\dots,N\}$ define the ancestral lineage indices as
\begin{equation}\label{eq:lineage}
b_n^i = i~~{\rm and}~~ b_k^i = a_k^{b_{k+1}^i}\, ,~~~ k\in\{0,\dots,n-1\} \, .
\end{equation}
%Then we have the estimate
%$$
%\sum_{i=1}^Nw_{n,\alpha,\theta}^i\varphi(x_{0:n}^{b_{0:n}^i}(1:k_{\alpha})).
%$$
%As it will prove useful in the next section, one can consider sampling from
%One can sample from
{For each $i=1,\dots,N$,
define
\begin{equation}\label{eq:smoothingsample}
\bar{x}_{0:n}^{i}(1:k_{\alpha}) := x_{0:n}^{b_{0:n}^i}(1:k_{\alpha}).
\end{equation}}
The following empirical measure then provides an approximation of
$\xi_{n,\alpha(1:k_{\alpha}),\theta}(x_{0:n}(1:k_{\alpha}))$
\begin{equation}\label{eq:target_approx}
\sum_{i=1}^Nw_{n,\alpha,\theta}^i\delta_{{\bar{x}_{0:n}^{i}(1:k_{\alpha})}}(dx_{0:n}) \, .
\end{equation}
%to .
This will prove useful in the next section.

The normalization constant
$Z_{n,\alpha,\theta}= \int_{\bigotimes_{i=1}^{k_{\alpha}}\mathsf{X}_{\alpha(i)}^{n+1}} \xi_{n,\alpha(1:k_{\alpha}),\theta}(x_{0:n}(1:k_{\alpha})) d x_{0:n}(1:k_{\alpha}) $ %, %call it $
can be unbiasedly estimated \cite{delm:04} by
\begin{equation}\label{eq:nc_est}
Z_{n,\alpha,\theta}^N = \prod_{p=0}^n \Big(\frac{1}{N}\sum_{i=1}^N \check{g}_{\theta,\alpha(1:k_{\alpha})}(x_p^i(1:k_{\alpha}),y_p)\Big).
\end{equation}

It is noted that %whilst
particle filters often do not work well in high-dimensions (e.g.~\cite{snyder}).
{However, in some cases, where the target probability is a high and finite dimensional discretization of
an infinite dimensional distribution}
(as will be the case in the context of this article), the algorithm can work quite well; see e.g.~\cite{kantas}.

\subsection{Particle MCMC}
\label{sec:pmcmc}

\blu{
In this section, we focus on approximating
expectations w.r.t.~$\xi_{n,\alpha(1:k_{\alpha})}(x_{0:n}(1:k_{\alpha}),\theta)$, {for a fixed $n$}.
{The notation ${\bar{x}}_{0:n}^{(i)}(1:k_{\alpha}),\theta^i$ 
will refer to the $i^{\textrm{th}}$ sample of a Markov chain designed to approximate expectations
w.r.t.~$\xi_{n,\alpha(1:k_{\alpha})}$. In the algorithms to be presented, $r(\theta^{i-1},\cdot)$
is a proposal density on $\Theta$ which we will assume is a postive probability density w.r.t.~$d\theta$ for any $\theta^{i-1}$.}
{In Algorithm \ref{alg:pmcmc} %(resp.~Algorithm \ref{alg:pmcmc}) 
we present an approach to approximate expectations w.r.t. %~$\xi_{n,\alpha(1:k_{\alpha})}(\theta)$ (resp.~
$\xi_{n,\alpha(1:k_{\alpha})}(x_{0:n}(1:k_{\alpha}),\theta)$.}
The method in Algorithm \ref{alg:pmcmc} %(resp.~Algorithm \ref{alg:pmcmc}) 
is called particle MCMC (PMCMC). % (resp.~PMCMC).}
}

%
%\begin{algorithm}[!b]
%\caption{Marginal PMCMC Algorithm}
%\label{alg:pmcmcm}
%%\begin{algorithmic}
%\begin{itemize}
%\item \textbf{Initialize:} Set $i=0$ and sample $\theta^0$ from the prior.
%Given $\theta^0$ run the particle filter in Algorithm \ref{algo:pf}
%and record the estimate of
%$Z_{n,\alpha,\theta^0}^N$
%from eq.~\eqref{eq:nc_est}. 
%%Select a trajectory $x_{0:n}^i(1:k_{\alpha})$ from the particle filter just run using \eqref{eq:target_approx}, denote the stored state $x_{0:n}^0(1:k_{\alpha})$.}
%\item \textbf{Iterate:}
%\begin{itemize}
%\item (I) Set $i=i+1$ and propose $\theta'$ given $\theta^{i-1}$ from a proposal $r(\theta^{i-1},\cdot)$ (described in the main text).
%\item (II) Given $\theta'$ run the particle filter in Algorithm \ref{algo:pf}
%and record the estimate $Z_{n,\alpha,\theta'}^N$.
%%Select a trajectory $x_{0:n}^{s'}(1:k_{\alpha})$ from the particle filter just run using \eqref{eq:target_approx}.
%%(y_{1:k})$.
%\item (III) Set $\theta^i=\theta'$ %, $x_{0:n}^i(1:k_{\alpha})=x_{0:n}^{s'}(1:k_{\alpha})$
%with probability
%$$
%\min\Big\{
%1,
%\frac{Z_{n,\alpha,\theta'}^N%(y_{1:k})
%\nu(\theta')r(\theta',\theta^{i-1})}
%{Z_{n,\alpha,\theta^{i-1}}^N
%\nu(\theta^{i-1})r(\theta^{i-1},\theta')}
%\Big\}
%$$
%otherwise $\theta^i=\theta^{i-1}$. %, $x_{0:n}^{i}(1:k_{\alpha})=x_{0:n}^{i-1}(1:k_{\alpha})$.
%\end{itemize}
%\end{itemize}
%\end{algorithm}

\begin{algorithm}[!h]
\blu{\caption{PMCMC Algorithm}
\label{alg:pmcmc}
%\begin{algorithmic}
\begin{itemize}
\item \textbf{Initialize:} 
\begin{itemize}
\item[(0)] Set $i=0$ and sample $\theta^0$ from the prior.
Given $\theta^0$ run the particle filter in Algorithm \ref{algo:pf}
and record the estimate of
$Z_{n,\alpha,\theta^0}^N$
from eq.~\eqref{eq:nc_est}. 
\item[(I)] Select a trajectory $x_{0:n}^i(1:k_{\alpha})$ from the particle filter just run using \eqref{eq:target_approx},
 denote the stored state $\overline x_{0:n}^{(0)}(1:k_{\alpha})$.
\end{itemize}
%Select a trajectory $x_{0:n}^i(1:k_{\alpha})$ from the particle filter just run using \eqref{eq:target_approx}, denote the stored state $x_{0:n}^0(1:k_{\alpha})$.}
\item \textbf{Iterate:}
\begin{itemize}
\item[(II)] Set $i=i+1$ and propose $\theta'$ given $\theta^{i-1}$ from a proposal $r(\theta^{i-1},\cdot)$ (described in the main text).
\item[(III)] Given $\theta'$ run the particle filter in Algorithm \ref{algo:pf}
and record the estimate $Z_{n,\alpha,\theta'}^N$.
\item[(IV)] Select a trajectory $x_{0:n}^{s'}(1:k_{\alpha})$ from the particle filter just run using \eqref{eq:target_approx}.
%Select a trajectory $x_{0:n}^{s'}(1:k_{\alpha})$ from the particle filter just run using \eqref{eq:target_approx}.
%(y_{1:k})$.
\item[(V)] Set $\theta^i=\theta'$ %, $x_{0:n}^i(1:k_{\alpha})=x_{0:n}^{s'}(1:k_{\alpha})$
with probability
$$
\min\Big\{
1,
\frac{Z_{n,\alpha,\theta'}^N%(y_{1:k})
\nu(\theta')r(\theta',\theta^{i-1})}
{Z_{n,\alpha,\theta^{i-1}}^N
\nu(\theta^{i-1})r(\theta^{i-1},\theta')}
\Big\}
$$
otherwise $\theta^i=\theta^{i-1}$. %, $x_{0:n}^{i}(1:k_{\alpha})=x_{0:n}^{i-1}(1:k_{\alpha})$.
\item[(VI)] Let $\overline x_{0:n}^{(i)}(1:k_{\alpha})=x_{0:n}^{s'}(1:k_{\alpha})$ if $\theta^i=\theta'$,
otherwise let $\overline x_{0:n}^{(i)}(1:k_{\alpha})=\overline x_{0:n}^{(i-1)}(1:k_{\alpha})$ if $\theta^i=\theta^{i-1}$.\end{itemize}
\end{itemize}}
%
%\begin{itemize}
%\item \textbf{Initialize}. As in Algorithm \ref{alg:pmcmcm}.
%\item \textbf{Iterate}:
%\begin{itemize}
%\item (I-II) as in Algorithm \ref{alg:pmcmcm}.
%\item (II.b) %(y_{1:k})$.
%\item (III) as in Algorithm \ref{alg:pmcmcm}.
%%Set $\theta^i=\theta'$,
%\item (III.b) Let $\overline x_{0:n}^i(1:k_{\alpha})=x_{0:n}^{s'}(1:k_{\alpha})$ if $\theta^i=\theta'$,
%otherwise let $\overline x_{0:n}^{i}(1:k_{\alpha})=\overline x_{0:n}^{i-1}(1:k_{\alpha})$ if $\theta^i=\theta^{i-1}$.
%\end{itemize}
%\end{itemize}
\end{algorithm}
%It is shown in \cite{andrieu}  %show that to estimate
%that one has the following consistent estimator of
\blu{{The target density associated to the PMCMC kernel described in Algorithm \ref{alg:pmcmc}
on the state-space 
$\Theta\times \Big(\bigotimes_{i=1}^{k_{\alpha}}\mathsf{X}_{\alpha(i)}^{n+1}\Big)^N\times\{1,\dots,N\}^{Nn+1}$ %\times\{1,\dots,N\}$
is given by
\begin{equation*}%\label{eq:andrieu}%
\tilde{\xi}_{n,\alpha(1:k_{\alpha})}(x_{0:n}^{1:N}(1:k_{\alpha}),a_{0:n-1}^{1:N},\theta,s)   
:= \frac1{Z_{n,\alpha}} w_{n,\alpha,\theta}^s Z_{n,\alpha,\theta}^N \psi_{\alpha,\theta}(x_{0:n}^{1:N}(1:k_{\alpha}),a_{0:n-1}^{1:N}) \, ,
\end{equation*}
where we recall that $\psi_{\alpha,\theta}$ is the density of the particle filter \eqref{eq:pfjoint}, 
and $Z_{n,\alpha} = \int_{\Theta} Z_{n,\alpha, \theta}$ is the normalizing constant, 
due to the unbiased property of \eqref{eq:nc_est}.
This is shown in \cite[Theorem 4]{andrieu}, as well as the fact that 
%\Big(\check{g}_{\theta,\alpha(1:k_{\alpha})}(x_p^s(1:k_{\alpha}),y_p)\Big)
%\Big(\sum_{j=1}^N \check{g}_{\theta,\alpha(1:k_{\alpha})}(x_p^j(1:k_{\alpha}),y_p) \Big)^{-1}\\
\begin{eqnarray}\nonumber
& \tilde{\xi}_{n,\alpha(1:k_{\alpha})}(x_{0:n}^{1:N}(1:k_{\alpha}),a_{0:n-1}^{1:N},\theta,s) \\
& =  \frac{\xi_{n,\alpha(1:k_{\alpha})}(\bar{x}_{0:n}^{s}(1:k_{\alpha}),\theta)}{N^{n+1}}
\frac{\psi_{\alpha,\theta}(x_{0:n}^{1:N}(1:k_{\alpha}),a_{0:n-1}^{1:N})}
{\check{\mu}_{\theta,\alpha(1:k_{\alpha})}(x_0^{b_0^s}(1:k_{\alpha}))
\prod_{p=1}^n
w_{p-1,\alpha,\theta}^{b_{p-1}^s} \check{f}_{\theta,\alpha(1:k_{\alpha})}(x_{p-1}^{b_{p-1}^s}(1:k_{\alpha}),x_{p}^{b_{p}^s}(1:k_{\alpha}))} \, ,
%\end{split}\]
\label{eq:pmcmc_density}
\end{eqnarray}
where $b_p^s$ is defined in \eqref{eq:lineage}. %\cite{andrieu} show that 
In other words $(\bar{x}_{0:n}^{s}(1:k_{\alpha}),\theta)$ has marginal density $\xi_{n,\alpha(1:k_{\alpha})}$.}
{As a result, consider estimating the integral
$$\int_{\bigotimes_{i=1}^{k_{\alpha}}\mathsf{X}_{\alpha(i)}^{n+1}\times\Theta} %\varphi_{\alpha(j)}
{\Phi}(x_{0:n}(1:k_{\alpha}),\theta)\xi_{n,\alpha(1:k_{\alpha})}(x_{0:n}(1:k_{\alpha}),\theta)dx_{0:n}(1:k_{\alpha})d\theta$$
for an integrable and real-valued function 
{$\Phi: \bigotimes_{i=1}^{k_{\alpha}}\mathsf{X}_{\alpha(i)}^{n+1}\times\Theta \rightarrow \bbR$. }
%\varphi_{\alpha(j)}$ and any $j\in\{1,\dots,k_{\alpha}\}$.
\cite{andrieu} show that, for Algorithm \ref{alg:pmcmc}, this above quantity is consistently estimated (that is, the estimate converges almost surely as $N\rightarrow+\infty$) by:}
\begin{equation}\label{eq:pmcmcest}
\frac{1}{N}\sum_{i=1}^N{\Phi}
%\varphi_{\alpha(j)}
({\bar{x}}_{0:n}^{(i)}(1:k_{\alpha}),\theta^i) \, .
\end{equation}
As a result, for $n$ fixed one can approximate the r.h.s.~of \eqref{eq:basic_idea} as,
%\blu{
$$
%\Delta\mathbb{E}^{N_{\alpha}}_{\pi_{n,\alpha}}[\varphi_{\alpha}(X_{0:n},\theta)] := 
\sum_{l=1}^{k_{\alpha}'}%(-1)^{|\alpha(k_{\alpha})-\alpha(2i)|}
\tau_{l,\alpha}
\Bigg(\frac{
\sum_{i=1}^N\varphi_{\alpha(2l)}(\overline{x}_{0:n}^{(i)}(l),{\theta}^i)
H_{2l,n,\alpha,\theta}(\overline{x}_{0:n}^{(i)}(1:k_\alpha),{\theta}^i)
}{\sum_{i=1}^N H_{2l,n,\alpha,\theta}(\overline{x}_{0:n}^{(i)}(1:k_\alpha),{\theta}^i)} -
\frac{
\sum_{i=1}^N \varphi_{\alpha(2l-1)}(\overline{x}_{0:n}^{(i)}(l),{\theta}^i)
H_{2l-1,n,\alpha,\theta}(\overline{x}_{0:n}^{(i)}(1:k_\alpha),{\theta}^i)
}{\sum_{i=1}^N H_{2l-1,n,\alpha,\theta}(\overline{x}_{0:n}^{(i)}(1:k_\alpha),{\theta}^i)}\Bigg) \, ,
$$
where we remind the reader that {$\tau_{i,\alpha}=(-1)^{|\alpha(k_{\alpha})-\alpha(2i)|}$}.
We hence refer to using Algorithm \ref{alg:pmcmc} 
in the context of \eqref{eq:basic_idea} within \eqref{eq:mimc} as above as MIPMCMC.
{Note that steps (I), (IV), (VI) can be ignored if one is only interested in estimation related to $\theta$.
%the target density of the Algorithm \ref{alg:pmcmcm} is t
The resulting chain has the marginal of \eqref{eq:pmcmc_density} as its target:
$\tilde{\xi}_{n,\alpha(1:k_{\alpha})}(x_{0:n}^{1:N}(1:k_{\alpha}),a_{0:n-1}^{1:N},\theta)$ (\cite{andrieu, chopin2}).}
In the  context of M|-PMCMC one must compute the weights in the expression above regardless, 
and so the joint distribution is required.}

\subsection{SMC$^{2}$}\label{sec:smc2}

%\blu{This method %ology is essentially
%entails constructing an outer SMC sampler algorithm \cite{chopin2002sequential} on the %an
%extended state-space, %.
%%The algorithm
%which targets the \emph{marginals} $\tilde{\xi}_{n,\alpha(1:k_{\alpha})}(x_{0:n}^{1:N}(1:k_{\alpha}),a_{0:n-1}^{1:N},\theta)$, %, of the PMMH method, i.e. that
%%method defined in % in
%i.e.~the targets of
%Algorithm \ref{alg:pmcmcm}. %, but can also, target
%Of course we require samples from the joint
%$\tilde{\xi}_{n,\alpha(1:k_{\alpha})}(x_{0:n}^{1:N}(1:k_{\alpha}),a_{0:n-1}^{1:N},\theta,s)$
%in order to approximate the terms of \eqref{eq:basic_idea}.  We explain how to do this below.}
%To achieve this one simply uses the full Algorithm \ref{alg:pmcmc} for the mutation at the final step $n$.}

To consider the method to be discussed, we start with some definitions.
We define the spaces:
\begin{eqnarray*}
\mathsf{E}_0 & := & \Big(\bigotimes_{i=1}^{k_{\alpha}}\mathsf{X}_{\alpha(i)}\Big)^N\times\Theta \\
\mathsf{E}_n & := & \Big(\bigotimes_{i=1}^{k_{\alpha}}\mathsf{X}_{\alpha(i)}^{n+1}\Big)^N\times\{1,\dots,N\}^{nN}\times\Theta \quad n\geq 1
\end{eqnarray*}
and states
\begin{eqnarray*}
U_{0,\alpha} & := & (X_0^{1:N}(1:k_{\alpha}),\theta) \\
U_{n,\alpha} & := & (X_{0:n}^{1:N}(1:k_{\alpha}),a_{0:n-1}^{1:N},\theta) \quad n\geq 1
\end{eqnarray*}
and note $U_{0,\alpha}\in\mathsf{E}_0$, $U_{n,\alpha}\in\mathsf{E}_n$, $n\geq 1$.

{We now introduce the SM$\textrm{C}^2$ method in Algorithm \ref{alg:smc2}.
This method is %approach is essentially just
a type of particle filter which targets {the} sequence of probability distributions
$\{\tilde{\xi}_{n,\alpha(1:k_{\alpha})}(x_{0:n}^{1:N}(1:k_{\alpha}),a_{0:n-1}^{1:N},\theta)\}_{n\geq 0}$ (see \cite[Proposition 1]{chopin2} for a justification) which admit
$\{\xi_{n,\alpha(1:k_{\alpha})}(\theta)\}_{n\geq 0}$ as a particular marginal; we explain how one can estimate expectations w.r.t.~$\xi_{n,\alpha(1:k_{\alpha})}(x_{0:n}(1:k_{\alpha}),\theta)$ below. The convergence (as $N_{\alpha}\rightarrow\infty$) of such an algorithm, then follows the theory of particle approximations of 
Feynman-Kac formulae, as described in \cite{delm:04}.

\begin{algorithm}[!b]
\caption{An  SMC$^{2}$ Algorithm 
{targeting $\tilde{\xi}_{n,\alpha(1:k_{\alpha})}(x_{0:n}^{1:N}(1:k_{\alpha}),a_{0:n-1}^{1:N},\theta)$}}
\label{alg:smc2}
%\begin{algorithmic}
\begin{itemize}
\item \textbf{Initialize}. Set $n=0$, for $i\in\{1,\dots,N_{\alpha}\}$ 
sample $\theta^i$ from the prior $\nu$, 
and $X_0^{i,j}(1:k_{\alpha})$ from $\check{\mu}_{\theta^i,\alpha(1:k_{\alpha})}(\cdot)$, $j\in\{1,\dots,N\}$. 
Compute the weight:
$$
G_{0,\alpha}(u_{0,\alpha}^i) = \frac{1}{N}\sum_{j=1}^N \check{g}_{\theta^i,\alpha(1:k_{\alpha})}(x_0^{i,j}(1:k_{\alpha}),y_0).
$$
 
\item \textbf{Iterate}:
\begin{itemize}
\item[(I)] {\bf Select}: Set $n=n+1$, resample $u_{n-1,\alpha}^{1:N_\alpha}$
using the normalized $\{G_{n-1,\alpha}(u_{n-1,\alpha}^i)\}_{i=1}^{N_\alpha}$,
denoting the resulting samples $\hat{u}_{n-1,\alpha}^{1:N_\alpha}$.
\item[(II)] {\bf Mutate}:
For $i\in\{1,\dots,N_{\alpha}\}$ generate  $\tilde{u}_{n-1,\alpha}^{i}|\hat{u}_{n-1,\alpha}^{i}$ using one iteration of Algorithm \ref{alg:pmcmc} \blu{(here we only require samples from the marginal
$\tilde{\xi}_{n,\alpha(1:k_{\alpha})}(x_{0:n}^{1:N}(1:k_{\alpha}),a_{0:n-1}^{1:N},\theta)$ and so steps
(I), (IV), (VI) of Algorithm \ref{alg:pmcmc} can be ignored)}. %, except the sampling of the trajectory.
\item[(III)] {\bf Extend}: For $i\in\{1,\dots,N_{\alpha}\}$ sample $X_n^{i,j}(1:k_{\alpha}),a_{n-1}^{i,j}$, $j\in\{1,\dots,N\}$ from
$$
\prod_{j=1}^N \frac{\check{g}_{\tilde{\theta}^i,\alpha(1:k_{\alpha})}(\tilde{x}_{n-1}^{i,a_{n-1}^{i,j}}(1:k_{\alpha}),y_{n-1})}{\sum_{l=1}^N\check{g}_{\tilde{\theta}^i,\alpha(1:k_{\alpha})}(\tilde{x}_{n-1}^{i,l}(1:k_{\alpha}),y_{n-1})}\check{f}_{\tilde{\theta}^i,\alpha(1:k_{\alpha})}(\tilde{x}_{n-1}^{i,a_{n-1}^{i,j}}(1:k_{\alpha}),x_{n}^{i,j}(1:k_{\alpha})) \, .
$$
Set $u_{n,\alpha}^i = (\tilde{u}_{n-1,\alpha}^{i},x_n^{i,1:N}(1:k_{\alpha}),a_{n-1}^{i,1:N})$\, .
\item[(IV)] {\bf Compute the weight}: For $i\in\{1,\dots,N_{\alpha}\}$
$$
G_{p,\alpha}(u_{n,\alpha}^i) = \frac{1}{N}\sum_{j=1}^N \check{g}_{\theta^i,\alpha(1:k_{\alpha})}(x_n^{i,j}(1:k_{\alpha}),y_n) \, .
$$
\end{itemize} 
\end{itemize}
\end{algorithm}

\subsubsection{Intuition of the Algorithm}

To understand the intution of the algorithm of \cite{chopin2}, 
consider the first step, where the objective is to approximate expectations w.r.t. 
$$
\tilde{\xi}_{0,\alpha(1:k_{\alpha})}(x_{0}^{1:N}(1:k_{\alpha}),\theta) \propto
\Big(\frac{1}{N}\sum_{j=1}^N \check{g}_{\theta,\alpha(1:k_{\alpha})}(x_0^{j}(1:k_{\alpha}),y_0)\Big)
\Big(\prod_{j=1}^N \check{\mu}_{\theta^i,\alpha(1:k_{\alpha})}(x_{0}^{j}(1:k_{\alpha}))\Big) \nu(\theta).
$$
This can be achieved by (self-normalized) importance sampling, just as in the initialization step of Algorithm \ref{alg:smc2}. This is because the term $G_{0,\alpha}(u_{0,\alpha}) $ is an importance weight, which allows one to correct for the discrepancy between the distribution sampled (in the initialization step of Algorithm \ref{alg:smc2}) and the one of interest.

We now want to move our samples, in such a way as to approximate expectations
w.r.t.~$\tilde{\xi}_{1,\alpha(1:k_{\alpha})}(x_{0:1}^{1:N}(1:k_{\alpha}),a_{0}^{1:N},\theta)$.
This can be achieved in the iterate step of Algorithm \ref{alg:smc2}, as we now explain. In the select step, this is a resampling of the samples, just as in the particle filter. The resulting
samples are approximately sampled from $\tilde{\xi}_{0,\alpha(1:k_{\alpha})}(x_{0}^{1:N}(1:k_{\alpha}),\theta)$. The mutate step will now produce new samples which are still approximately sampled from $\tilde{\xi}_{0,\alpha(1:k_{\alpha})}(x_{0}^{1:N}(1:k_{\alpha}),\theta)$, as the transition kernel leaves this probability invariant. 
The extend step now
produces the additional random variables 
needed to approximate expectations w.r.t.~$\tilde{\xi}_{1,\alpha(1:k_{\alpha})}(x_{0:1}^{1:N}(1:k_{\alpha}),a_{0}^{1:N},\theta)$.
The strategy employed leads to the convenient weight function $G_{1,\alpha}(u_{1,\alpha})$ in the next step. 
This corresponds to the ratio, up-to a normalizing constant, of
$\tilde{\xi}_{1,\alpha(1:k_{\alpha})}(x_{0:1}^{1:N}(1:k_{\alpha}),a_{0}^{1:N},\theta)$ to 
the product of $\tilde{\xi}_{0,\alpha(1:k_{\alpha})}(x_{0}^{1:N}(1:k_{\alpha}),\theta)$ %multipled by the proposal mechanism in 
the proposal used in the extend step. 
Expectations w.r.t.~$\tilde{\xi}_{1,\alpha(1:k_{\alpha})}(x_{0:1}^{1:N}(1:k_{\alpha}),a_{0}^{1:N},\theta)$ can now be approximated again by
self-normalized importance sampling. The algorithm then just continues for the rest of the sequence $\{\tilde{\xi}_{n,\alpha(1:k_{\alpha})}(x_{0:n}^{1:N}(1:k_{\alpha}),a_{0:n-1}^{1:N},\theta)\}_{n\geq 2}$.

The reason {why} one considers the sequence
$\{\tilde{\xi}_{n,\alpha(1:k_{\alpha})}(x_{0:n}^{1:N}(1:k_{\alpha}),a_{0:n-1}^{1:N},\theta)\}_{n\geq 0}$,
instead of the original
$\{\xi_{n,\alpha(1:k_{\alpha})}(x_{0:n}(1:k_{\alpha}),\theta)\}_{n\geq 0}$, is because 
the algorithm associated  with the former is expected to be more efficient than a related SMC algorithm for the latter. {To explain further, one can consider the naive algorithm which samples the initial $\theta$ from the prior (i.e.~$N$ samples) and then run a particle filter with
$N$ associated trajectories $x_{0:n}(1:k_{\alpha})$. The main issue here is of course that one never updates the $\theta$, so that estimates of expectations associated to $\theta$ would be very poor. This is further exacerbated by the path degeneracy problem for particle filters; one does not update the trajectory in the past, and due to the resampling operation the distinctness of the trajectories in the past will be essentially lost.
%{In other words, for any $N$, and large enough $n$, there will always be one particle at time $0$:
%$x_{0}(1:k_{\alpha})^j = x_{0}(1:k_{\alpha})^1$ for $j=1,\dots, N$.}
These latter issues can be circumvented by applying an MCMC kernel of invariant measure $\xi_{n,\alpha(1:k_{\alpha})}(x_{0:n}(1:k_{\alpha}),\theta)$ at each time step to each sample. {However, as we have remarked,
in general PMCMC is considered to be more efficient than this. }}
\blu{The choice of $N$ is an important tuning parameter, which is considered in detail in 
the work \cite{doucet2015efficient}. The recommendation there is $N \propto n$. 
See \cite{chopin2,kantas1} for further insights. 
We do not consider biased methods such as \cite{liu2001combined} here.}
}

%%Returning to Equation \eqref{eq:nc_est}, %the key point of the algorithm is we
%{WRONG: The following non-intuitive identity is crucial
%%ratio $
%%\tilde{\xi}_{p,\alpha(1:k_{\alpha})}(x_{0:p}^{1:N}(1:k_{\alpha}),a_{0:p-1}^{1:N},\theta)/\tilde{\xi}_{p-1,\alpha(1:k_{\alpha})}(x_{0:p-1}^{1:N}(1:k_{\alpha}),a_{0:p-2}^{1:N},\theta)$
%%normalization constants of
%%$\xi_{p,\alpha(1:k_{\alpha}),\theta}(x_{0:p}(1:k_{\alpha}))$ and $\xi_{p-1,\alpha(1:k_{\alpha}),\theta}(x_{0:p-1}(1:k_{\alpha}))$,
%%$Z_{p,\alpha,\theta}/Z_{p-1,\alpha,\theta}$ can be unbiasedly estimated by
%\begin{equation}\label{eq:ncrat_est}
%\frac{\tilde{\xi}_{p,\alpha(1:k_{\alpha})}(x_{0:p}^{1:N}(1:k_{\alpha}),a_{0:p-1}^{1:N},\theta)}{\tilde{\xi}_{p-1,\alpha(1:k_{\alpha})}(x_{0:p-1}^{1:N}(1:k_{\alpha}),a_{0:p-2}^{1:N},\theta)} =
%G_{p,\alpha}(u_{p,\alpha}) := \Big(\frac{1}{N}\sum_{i=1}^N \check{g}_{\theta,\alpha(1:k_{\alpha})}(x_p^i(1:k_{\alpha}),y_p)\Big).
%\end{equation}}
%So there is a natural mechanism by which to wrap an SMC sampler \cite{delm:06b}
%around PMCMC, %or embed PMCMC inside an SMC sampler if you prefer,
%resulting in an implementable sequential algorithm.
%{CAN WE EXPLAIN CONCISELY (AND CORRECTLY) THE NATURAL MECHANISM.... ?}
%The procedure is given in Algorithm \ref{alg:smc2}.

%In Algorithm \ref{alg:smc2} denote the equally weighted $N_{\alpha}-$empirical measure of $\tilde{u}_{p,\alpha}^{1:N_{\alpha}}$ as $\tilde{\eta}_{p,\alpha}^{N_{\alpha}}$.
%Now a

\subsubsection{Estimating Expectations with respect to the joint target}

\blu{We now concisely describe how to use the samples $\tilde{u}_{p,\alpha}^{1:N_{\alpha}}$
in order to estimate expectations with respect to the joint (coupled) state and parameter,
using the SMC analogue of the PMCMC estimator \eqref{eq:pmcmcest},
hence enabling estimation of \eqref{eq:basic_idea}. As described in 
Section \ref{sec:pmcmc} and Algorithm \ref{alg:pmcmc}, we require samples from 
${\tilde{\xi}_{n,\alpha(1:k_{\alpha})}(x_{0:n}^{1:N}(1:k_{\alpha}),a_{0:n-1}^{1:N},\theta,s)}$ to achieve this. 
As suggested in Algorithm \ref{alg:smc2} step (II),
assume for the moment that we have ignored steps (I), (IV), (VI) of Algorithm \ref{alg:pmcmc}, 
so for each SMC sampler particle we now need to sample $S$ 
and construct $\overline x_{0:n}^s(1:k_{\alpha})$.
Of course the full PMCMC Algorithm \ref{alg:pmcmc} could be used, in which case no further work is required.
However, since $s$ and $\overline x_{0:n}^s(1:k_{\alpha})$ do not appear in the recursion, 
and are not required until one estimates a joint expectation, 
this is considered as a separate step here.}
%Recall equations \eqref{eq:lineage}, \eqref{smoothings}, and \eqref{eq:target_app}}

\blu{At any time $n$, sample $S_n^i\in\{1,\dots,N\}$, $i\in\{1,\dots,N_{\alpha}\}$ with probability
\begin{equation}\label{eq:s_def}
\mathbb{P}(S_n^i=j|\tilde{u}_{n,\alpha}^i) = \frac{\check{g}_{\tilde{\theta}^i,\alpha(1:k_{\alpha})}(\tilde{x}_{n}^{i,j}(1:k_{\alpha}),y_n)}{\sum_{l=1}^N \check{g}_{\tilde{\theta}^i,\alpha(1:k_{\alpha})}(\tilde{x}_{p}^{i,l}(1:k_{\alpha}),y_n)}.
\end{equation}
{For $p{\leq}n$ recall the definition \eqref{eq:lineage}, and define
\begin{equation}\label{eq:vee}
%v_{p}^i(1:k_\alpha) 
\overline{x}_p^{(i)}(1:k_\alpha):= {\tilde x}^{i,b_{n}^{s^i_n}}_{p}(1:k_\alpha) \, .
\end{equation}}
\blu{Note this construction corresponds to sampling from \eqref{eq:target_approx},
as is done in Algorithm \ref{alg:pmcmc}. 
The notation $\overline{x}_p^{(i)}(1:k_\alpha)$ (which is redundant with the PMCMC notation)
has been used on purpose.
Now an SMC consistent estimator can be constructed analogous to the PMCMC estimator 
given in \eqref{eq:pmcmcest}, using the following empirical measure 
%Denote the augmented empirical measure of
%$(\tilde{u}_{p,\alpha}^{1:N_{\alpha}},s_p^{{1:N_\alpha}})$ as
{$$
\eta_{n,\alpha}^{N_\alpha} :=
\frac{1}{N}\sum_{i=1}^{N_\alpha} \delta_{\overline{x}_{0:n}^{(i)}(1:k_\alpha)}
%v_{p}^i(1:k_\alpha)} %(\tilde{u}_{p,\alpha}^{l},s_p^{l})} 
\, .
$$}}
Let %$(i,j)\in\{1,\dots,k_{\alpha}\}^2$, 
$l\in\{1,\dots,k_{\alpha}\}$, 
$\varphi:\mathbb{N}_0^d\times\mathsf{X}^{n+1}\times\Theta\rightarrow\mathbb{R}$,
$H:\mathbb{N}_0^d\times\bigotimes_{i=1}^{k_{\alpha}}\mathsf{X}_{\alpha(i)}^{n+1}\times\Theta\rightarrow\mathbb{R}$ (measurable and integrable w.r.t.~${\xi}_{n,\alpha(1:k_{\alpha})}$).
\blu{Recall that following from \eqref{eq:pmcmc_density} one has}
 %one can consistently estimate
\begin{eqnarray*}
& \int_{\mathsf{E}_p\times\{1,\dots,N\}} \varphi_{\alpha(l)}({\bar{x}}_{0:n}^s({l}),\theta)
H_{\alpha}({\bar{x}}_{0:n}^s(1:k_{\alpha}),\theta)
\tilde{\xi}_{n,\alpha(1:k_{\alpha})}(x_{0:n}^{1:N}(1:k_{\alpha}),a_{0:n-1}^{1:N},\theta,s)d(x_{0:n}^{1:N}(1:k_{\alpha}),a_{0:n-1}^{1:N},\theta,s) \\
& =\int_{\bigotimes_{l=1}^{k_{\alpha}}\mathsf{X}_{\alpha(l)}^{n+1}\times\Theta}
\varphi_{\alpha(l)}(x_{0:n}({l}),\theta)
H_{\alpha}(x_{0:n}(1:k_{\alpha}),\theta)
{\xi}_{n,\alpha(1:k_{\alpha})}(x_{0:n}(1:k_{\alpha}),\theta)d(x_{0:n}(1:k_{\alpha}),\theta) \, .
\end{eqnarray*}
%Then for $i,j\in\{1,\dots,k_{\alpha}\}$ one can consistently estimate
%$$
%\int_{\mathsf{E}_p\times\{1,\dots,N\}}
%{f(x_{0:n}^s(1:k_\alpha),\theta)} \tilde{\xi}_{n,\alpha(1:k_{\alpha})}(x_{0:n}^{1:N}(1:k_{\alpha}),a_{0:n-1}^{1:N},\theta,s)d(x_{0:n}^{1:N}(1:k_{\alpha}),a_{0:n-1}^{1:N},\theta,s)
%%\varphi_{\alpha(i)}(x_{0:n}^s(j),\theta) \tilde{\xi}_{n,\alpha(1:k_{\alpha})}(x_{0:n}^{1:N}(1:k_{\alpha}),a_{0:n-1}^{1:N},\theta,s)d(x_{0:n}^{1:N}(1:k_{\alpha}),a_{0:n-1}^{1:N},\theta,s)
%$$
%$$
%=\int_{\bigotimes_{i=1}^{k_{\alpha}}\mathsf{X}_{\alpha(i)}^{n+1}\times\Theta}
%{f(x_{0:n}(1:k_\alpha),\theta)} {\xi}_{n,\alpha(1:k_{\alpha})}(x_{0:n}(1:k_{\alpha}),\theta)d(x_{0:n}(1:k_{\alpha}),\theta) \, ,
%%\varphi_{\alpha(i)}(x_{0:n}(j),\theta) {\xi}_{n,\alpha(1:k_{\alpha})}(x_{0:n}(1:k_{\alpha}),\theta)d(x_{0:n}(1:k_{\alpha}),\theta) \, ,
%$$
%Now, for each $i$, the trajectory constructed in \eqref{eq:vee} 
%is identical to ${\bar{x}}_{0:n}^s(1:k_{\alpha})$ for the $i^{\rm th}$ particle filter 
%(except for notational simplification),
%so o
One can now consistently estimate the above expectation analogous to \eqref{eq:pmcmcest}
using \eqref{eq:vee} as follows
{$$
\eta_{n,\alpha}^{N_{\alpha}}\big(\varphi_{\alpha(l)} {H_{\alpha}} \big) = 
\frac{1}{N_{\alpha}}\sum_{i=1}^{N_{\alpha}}
\varphi_{\alpha(l)}(\overline{x}_{0:n}^{(i)}(l),\tilde{\theta}^i)
H_{\alpha}(\overline{x}_{0:n}^{(i)}(1:k_\alpha),\tilde{\theta}^i)
%\varphi_{\alpha(i)}
 \, .
%(\tilde{x}_{0:n}^{l,s^{l}}(j),\tilde{\theta}^l).
$$}}
Hence we have the following estimate of \eqref{eq:basic_idea} {
$$
\Delta\mathbb{E}^{N_{\alpha}}_{\pi_{n,\alpha}}[\varphi_{\alpha}(X_{0:n},\theta)] := \sum_{i=1}^{k_{\alpha}'}%(-1)^{|\alpha(k_{\alpha})-\alpha(2i)|}
\tau_{i,\alpha}
\Bigg(\frac{
\eta_{n,\alpha}^{N_{\alpha}}(\varphi_{\alpha(2i)}H_{2i,n,\alpha,\theta})
}{\eta_{n,\alpha}^{N_{\alpha}}(H_{2i,n,\alpha,\theta})} -
\frac{
\eta_{n,\alpha}^{N_{\alpha}}(\varphi_{\alpha(2i-1)}H_{2i-1,n,\alpha,\theta})
}{\eta_{n,\alpha}^{N_{\alpha}}(H_{2i-1,n,\alpha,\theta})}\Bigg) \, ,
$$
where we remind the reader that {$\tau_{i,\alpha}=(-1)^{|\alpha(k_{\alpha})-\alpha(2i)|}$}.}
%where for instance
%$$
%\eta_{n,\alpha}^{N_{\alpha}}(\varphi_{\alpha(2i)}H_{2i,n,\alpha,\theta}) = \frac{1}{N_{\alpha}}\sum_{l=1}^{N_{\alpha}}\varphi_{\alpha(2i)}(\overline{\tilde{x}}_{0:n}^{l}(2i),\tilde{\theta}^l)
%H_{2i,n,\alpha,\tilde{\theta}^l}(\overline{\tilde{x}}_{0:n}^l(1:k_{\alpha})).
%$$

\section{Theoretical Results}\label{sec:theory}

We now consider the MISMC$^{2}$ procedure in the previous section,
\blu{however the results naturally extend to the static 
MIPMCMC method as well, which appears as a component of this method.}
We will analyze the variance of our MI method(s), under %and we make
the following assumptions.
\begin{enumerate}[label={(A\arabic*)}]
\item There exist $0< \underline{C}< \overline{C}<+\infty$
such that for every $\alpha\in\mathcal{I}$, $\theta\in\Theta$,
$(x,y)\in\mathsf{X}_{\alpha}\times\mathsf{Y}_{\alpha}$
$$
\underline{C} \leq g_{\theta,\alpha}(x,y) \leq \overline{C}.
$$\label{hyp:A}
\item

For every $n\geq 0$, $\varphi:\mathbb{N}_0^d\times\mathsf{X}^{n+1}\times\Theta\rightarrow\mathbb{R}$ bounded, every $\alpha\in\mathcal{I}$,
there exist a $C(\alpha(1:k_{\alpha}))$, with $\lim_{\min_{1\leq i\leq d}\alpha_i \rightarrow+\infty}C(\alpha(1:k_{\alpha}))=0$,
such that for any collection of scalar, bounded random variables $\beta(\alpha(1:k_{\alpha}),2i,2i-1)$,
$i\in\{1,\dots,k_{\alpha}'\}$ we have almost surely
$$
\sup_{(x_{0:n}(1:k_{\alpha}),\theta)\in(\bigotimes_{i=1}^{k_{\alpha}}\mathsf{X}_{\alpha(i)}^{n+1})\times\Theta}\Big|
\Big\{\sum_{i=1}^{k_{\alpha}'}\tau_{i,\alpha}%(-1)^{|\alpha(k_{\alpha})-\alpha(2i)|}\times
\beta(\alpha(1:k_{\alpha}),2i,2i-1)
\Big[
\varphi_{\alpha(2i)}(x_{0:n}(1:k_{\alpha}),\theta) -
$$
$$
\varphi_{\alpha(2i-1)}(x_{0:n}(1:k_{\alpha}),\theta)\Big]\Big\}
\Big| \leq
C(\alpha(1:k_{\alpha}))\sum_{i=1}^{k_{\alpha}'}|\beta(\alpha(1:k_{\alpha}),2i,2i-1)|^2.
$$\label{hyp:B}
\end{enumerate}
{We remind the reader again that $\tau_{i,\alpha}=(-1)^{|\alpha(k_{\alpha})-\alpha(2i)|}$}.

\ref{hyp:A} is a strong, but standard, assumption that has been used in the HMM literature, particularly in the SMC context; see for instance \cite{delm:04}.
\ref{hyp:B} is certainly non-standard and in general one would like to deduce under simpler hypotheses. In our efforts to achieve this, we have not found a suitable technical approach and leave this more involved analysis to future work.
%We have t
The following result is the culmination of our efforts.
The expectation below is w.r.t.~the randomness in the SMC$^{2}$ algorithm.
{
\begin{theorem}\label{thm:main}
Assume \ref{hyp:A}-\ref{hyp:B}.  Then for every $n\geq 0$, $\varphi:\mathbb{N}_0^d\times\mathsf{X}^{n+1}\times\Theta\rightarrow\mathbb{R}$ bounded and every $\alpha\in\mathcal{I}$,
there exist a $C(\alpha(1:k_{\alpha}))$, with $\lim_{\min_{1\leq i\leq d}\alpha_i \rightarrow+\infty}C(\alpha(1:k_{\alpha}))=0$, such that:
$$
\mathbb{E}\Big[\Big(\Delta\mathbb{E}^{N_{\alpha}}_{\pi_{n,\alpha}}[\varphi_{\alpha}(X_{0:n},\theta)]-\Delta \mathbb{E}_{\pi_{n,\alpha}}[\varphi_{\alpha}(X_{0:n},\theta)]\Big)^2\Big] \leq \frac{C(\alpha(1:k_{\alpha}))}{N_{\alpha}}
%\sum_{i=1}^{k_{\alpha}'}(-1)^{|\alpha(k_{\alpha})-\alpha(2i)|}
%\Bigg(\frac{
%\eta_{n,\alpha}^{N_{\alpha}}(\varphi_{\alpha(2i)}H_{2i,n,\alpha,\theta})
%}{\eta_{n,\alpha}^{N_{\alpha}}(H_{2i,n,\alpha,\theta})} -
%\frac{
%\eta_{n,\alpha}^{N_{\alpha}}(\varphi_{\alpha(2i-1)}H_{2i-1,n,\alpha,\theta})
%}{\eta_{n,\alpha}^{N_{\alpha}}(H_{2i-1,n,\alpha,\theta})}\Bigg)
$$
%$$
%-\Delta \mathbb{E}_{\pi_{n,\alpha}}[\varphi_{\alpha}(X_{0:n},\theta)]\Big)^2\Big] \leq \frac{C(\alpha(1:k_{\alpha}))}{N_{\alpha}}
%$$
and
$$
\bigg|\mathbb{E}\Big[\Delta\mathbb{E}^{N_{\alpha}}_{\pi_{n,\alpha}}[\varphi_{\alpha}(X_{0:n},\theta)]
-\Delta \mathbb{E}_{\pi_{n,\alpha}}[\varphi_{\alpha}(X_{0:n},\theta)]\Big]\bigg| \leq \frac{C(\alpha(1:k_{\alpha}))}{N_{\alpha}}.
%\sum_{i=1}^{k_{\alpha}'}(-1)^{|\alpha(k_{\alpha})-\alpha(2i)|}
%\Bigg(\frac{
%\eta_{n,\alpha}^{N_{\alpha}}(\varphi_{\alpha(2i)}H_{2i,n,\alpha,\theta})
%}{\eta_{n,\alpha}^{N_{\alpha}}(H_{2i,n,\alpha,\theta})} -
%\frac{
%\eta_{n,\alpha}^{N_{\alpha}}(\varphi_{\alpha(2i-1)}H_{2i-1,n,\alpha,\theta})
%}{\eta_{n,\alpha}^{N_{\alpha}}(H_{2i-1,n,\alpha,\theta})}\Bigg)
$$
%$$
%-\Delta \mathbb{E}_{\pi_{n,\alpha}}[\varphi_{\alpha}(X_{0:n},\theta)]\Big]\bigg| \leq \frac{C(\alpha(1:k_{\alpha}))}{N_{\alpha}}.
%$$
\end{theorem}
}
\begin{proof}
Follows directly from Lemma \ref{lem:1} and Proposition \ref{prop:1} in the appendix.
\end{proof}

It is noted that our bound depends upon the time parameter and $d$ and we do not address these aspects in our subsequent discussion.

%\sum_{i=1}^{k_{\alpha}'}(-1)^{|\alpha(k_{\alpha})-\alpha(2i)|}
%\Bigg(\frac{
%\eta_{n,\alpha}^{N_{\alpha}}(\varphi_{\alpha(2i)}H_{2i,n,\alpha,\theta})
%}{\eta_{n,\alpha}^{N_{\alpha}}(H_{2i,n,\alpha,\theta})} -
%\frac{
%\eta_{n,\alpha}^{N_{\alpha}}(\varphi_{\alpha(2i-1)}H_{2i-1,n,\alpha,\theta})
%}{\eta_{n,\alpha}^{N_{\alpha}}(H_{2i-1,n,\alpha,\theta})}\Bigg) \, .

\subsection{MIMC considerations}\label{ssec:mimc}

Define a multi-index estimator as{
\begin{equation*}%\label{eq:mimcmcest}
\widehat{\varphi}_{\cI}^{\rm MI} := \sum_{\alpha\in \cI}
\Delta\mathbb{E}^{N_{\alpha}}_{\pi_{n,\alpha}}[\varphi_{\alpha}(X_{0:n},\theta)].
\end{equation*}
\blu{Below the necessary assumptions are given, which are common for multi-index methods.}
$\textrm{Cost}(X_{\alpha})$ denotes the cost of sampling the discretized random variable $X_{\alpha}$.
Recall $C(\alpha(1:k_\alpha))$ appears in Theorem  \ref{thm:main} and Assumption \ref{hyp:B}.

\begin{ass}[MISMC$^2$ rates]\label{ass:mimc2}
 For every $n\geq 0$, there exists $C<+\infty$, $w_i, \beta_i, \gamma_i >0$
for $i=1,\dots, d$, such that for every $\alpha\in\mathbb{N}_0^d$:
\begin{itemize}
\item[{\rm (a)}] $\left|\Delta \mathbb{E}_{\pi_{n,\alpha}}[\varphi_{\alpha}(X_{0:n},\theta)] \right|
\leq C \prod_{i=1}^d 2^{-w_i\alpha_i}$;
\item[{\rm (b)}] $C(\alpha(1:k_\alpha)) \leq C \prod_{i=1}^d 2^{-\beta_i\alpha_i}$;
\item[{\rm (c)}] {\rm Cost}$(X_{\alpha}) \leq C \prod_{i=1}^d 2^{\gamma_i \alpha_i}.$
\footnote{To be precise, there would typically be different constants, 
but it obviously suffices to take the largest.}
\end{itemize}
\end{ass}
}
%The following holds.

{
%THIS PART IS NOT DONE. RESULTS NEED TO BE GENERALIZED TO $\cI$.
Before presenting the main MISMC$^2$ theorem, we need to introduce some
index sets, which relate to Assumption \ref{ass:mimc2}.
%For simplicity, we will constrain our attention to $\alpha \in\mathcal{I}_{ m_1:m_d}$
%where the
The tensor product index set is defined by
\begin{equation}\label{eq:tensor}
\mathcal{I}_{\alpha^*}:= \{\alpha=(\alpha_1,\dots,\alpha_d)\in\mathbb{N}_0^d:
\alpha_1\in\{0,\dots,\alpha^*_1\},\dots,\alpha_d\in\{0,\dots,\alpha^*_d\}\}\, .
\end{equation}
%This set is suboptimal, as one can obtain the same error estimates (with a smaller constant) with MLMC.
The total degree index set for $L \in \bbR_+$ and $\zeta \in \mathbb{R}^d_+$ is defined as
\begin{equation}\label{eq:tensor}
\mathcal{I}_{\zeta,L}^{\rm TD}:= \left \{\alpha \in\mathbb{N}_0^d:
\sum_{i=1}^d \alpha_i \zeta_i \leq L \right \}\, .
\end{equation}
%One can obtain the same error estimates
%(perhaps with a larger constant) with the
%smaller total degree index set
%$\mathcal{I}_{\zeta,L}^{\rm TD} \subset \mathcal{I}_{(L/\zeta_1,\dots, L/\zeta_d)}$,
%%({TD stands for total degree})
%which includes a simplex of indices $\alpha$ from the tensor product set
%$\mathcal{I}_{(L/\zeta_1,\dots,L/\zeta_d)}$.
It is suggested in Sec. 2.2 of \cite{mimc} (and verified in numerical examples) that
the optimal index set is given by $\mathcal{I}_{L}^{\rm TD}=\mathcal{I}_{\zeta^*,L}^{\rm TD}$,
where $z^*_i \propto \log(2)(w_i+(\gamma_i-\beta_i)/2)$ and $\sum_{i=1}^d z^*_i =1$.
%; in particular, those $\alpha$ such that $\alpha \cdot \zeta \leq L$, for some vector $\zeta \in \mathbb{R}^d_+$ and some $M>0$.
See \cite{mimc} for a detailed investigation of the various relationships between the rates of convergence
and these index sets.
The methodology developed is applicable to general index sets, but we will present the proposition 
for only a simplified set of circumstances in the interest of clarity and simplicity.
%and so the results can only be expected to improve.
%However, restricting attention to the set $\mathcal{I}_{m_1:m_d}$ will significantly simplify the presentation of the method.
}

{
\begin{prop}
\label{pro:mimcmc}
Assume \ref{hyp:A}-\ref{hyp:B}, Assumption \ref{ass:mimc2} and that
$\beta_i>\gamma_i$, for all $i=1,\dots, d$, and %.  Either
one of the following cases holds
\begin{itemize}
\item[{\rm [A]}]{$\cI = \mathcal{I}_{\alpha^*}$ and $\sum_{i=1}^d \gamma_i/w_i \leq 2$; or}
\item[{\rm [B]}]{$\cI = \mathcal{I}_{L}^{\rm TD}$.}
\end{itemize}
%If $\beta_i>\gamma_i$, for all $i=1,\dots, d$,
%and assuming %${\rm max}\{
%$\sum_{i=1}^d \gamma_i/w_i
%%,\sum_{i=1}^d\beta_i/w_i\}
%\leq 2$,
Then there exist $C<+\infty$,
and either $\alpha^*=(m_1,\dots,m_d)\in\mathbb{N}_0^d$ in case {\rm [A]} or $L \in \bbR_+$ in case {\rm [B]},
and $\{ N_\alpha \}_{\alpha \in \cI}$, such that for any $\varepsilon>0$:
$$
\bbE\left [ \left (\widehat{\varphi}_{\cI}^{\rm MI} - \mathbb{E}[\varphi(X_{0:n},\theta)] \right )^2 \right ]
\leq C \varepsilon^2 \ ,
$$
for a cost of $\cO(\varepsilon^{-2})$.%, with $\widehat{\varphi}_{\cI_{ m_1:m_d}}^{\rm MI}$ defined in \eqref{eq:mimcmcest}.
\end{prop}}
\begin{proof}
\blu{Standard optimization of cost as a function of $N_\alpha$ for a fixed variance 
yields that 
$$N_\alpha = \varepsilon^{-2} 
\left( \sum_{\alpha \in \cI}\sqrt{C(\alpha(1:k_\alpha)){\rm Cost}(X_\alpha)} \right)^{-1}
\sqrt{C(\alpha(1:k_\alpha))/{\rm Cost}(X_\alpha)}\, ,$$
where $C(\alpha(1:k_\alpha))$ and ${\rm Cost}(X_\alpha)$ are defined in Assumption \ref{ass:mimc2} (b-c).}
Under the assumptions above, and following from Theorem \ref{thm:main},
the proof {for case {\rm [A]} is the same as that of Proposition 3.2 in \cite{jklz}.
Case {\rm [B]} follows from Theorem 2.2 of \cite{mimc}
(see also Theorem 2 of \cite{giles1}).}
\end{proof}

\blu{Note that $\varepsilon^2$ here represents the asymptotic MSE and the proposition
above relates this to the cost. 
Simply put, the proposition states that the cost is proportional to the inverse of the MSE,
which is called the {\em canonical case} 
because it is the best one can expect from any Monte Carlo method.}
This can be readily generalized to different relationship between the coefficients $(w_i,\beta_i,\gamma_i)$.
{There are many different cases in general, but the rules of thumb are that
(i) the complexity has a logarithmic penalty if $\beta_j \leq \gamma_j$ for any $j$, and
(ii) there is a smaller exponent on $\varepsilon$ as well if $\beta_j < \gamma_j$ for any $j$.
%and importantly different index sets, such as the analogous total degree set $\cI^{\rm TD}_{m_1:m_d}$,
%which was mentioned above.
%These modifications will result in the same cost benefit here as in the work \cite{mimc},
%and this will be very important for practical application.
The various conditions can be derived in a similar manner as in
\cite{mimc} (see also \cite{giles1} and \cite{jklz} for some discussion).
Note that if $\sum_{i=1}^d \gamma_i/w_i > 2$ instead in case {\rm [A]} then the cost is
$\varepsilon^{-\sum_{i=1}^d \gamma_i/w_i}$,
corresponding to the cost of a single realization at the finest discretization.
In this case, the cost of MLMC will still be at least as large,
because a single realization at the finest discretization
of the tensor product index set will always be required.}

\section{Numerical Results}
\label{sec:numerics}
\subsection{Modelling} \label{ssec:set-up}
We illustrate the performance of the proposed methods on the Bayesian parameter inference problem of a partially observed stochastic system which is the solution to a SPDE. Comparisons are made with sampling from the most precise discretization of the underlying stochastic system using either PMCMC or $\text{SMC}^2$. {The objective here is to illustrate the theory and test the applicability of the method under weaker assumptions than provided by the theory.
Therefore, we will restrict attention to the total degree index set $\cI_\alpha$,
despite its suboptimality in this example.}

We consider the %semi-linear
stochastic heat equation on a one-dimensional domain $[0,1]$ over the time interval $[0,T]$, i.e.
\begin{equation*}
	\frac{\partial u}{\partial t} = \frac{\partial^2 u}{\partial x^2} + a u + \theta\dot{W_t}
\end{equation*}
with the Dirichlet boundary condition and initial value $u(x,0) = u_0(x) = \sum_{k=1}^{\infty}u_{k,0}\,e_k(x)$ for $x \in (0,1)$. The eigenfunction $e_k(x) = \sqrt{2}\sin(k\pi x)$ has the corresponding eigenvalue $\lambda_k = k^2\pi^2$ and the noise $W_t$ is the space-time white noise, i.e.~the cylindrical Brownian motion given by $W_t = \sum_{k=1}^{\infty}\sqrt{q_k}e_k\beta_t^k$, where $\beta_t^k$ $(k \geq 1)$ are i.i.d.~scalar Brownian motions. The hidden process is assumed to be modelled by the solution to this SPDE with $q_k = 1$ and $u_{k,0} = 1$ for {$1 \leq k \leq K_{\rm max}$ and $u_{k,0}=0$ otherwise}.

Pointwise observations of the process are obtained at times {$t(n) = n\delta$ for $n=1,2,...,100$ and $\delta = 0.001$}, and at the locations $x_1=1/3$ and $x_2=2/3$ under an additive Gaussian noise with mean zero and variance $\tau^2 = 1$. If we denote the observation vector at time $t(n)$ by $y_n = (y_{n,1},y_{n,2})^T$, the corresponding likelihood function is
\begin{equation*}
	g(x_n,y_n) \propto \prod_{i=1}^{2}\text{exp}\left(-\frac{1}{2\tau^2}\left(y_{n,i}-u(x_i,t(n))\right)^2\right)
\end{equation*}
where $u(x_i,t(n))$ is the solution of the above SPDE at time $t(n)$ and location $x_i$ and note that $u(x,t) = \sum_{k=1}^{\infty}u_{k,t}\,e_k(x)$. The model parameter $\theta$ is assumed to be unknown and is assigned a prior distribution $\text{Gamma}(1,\sqrt{0.1})$ where $\text{Gamma}(a,b)$ represents the Gamma distribution with shape parameter $a$ and scale parameter $b$. 
A fixed sequence of observations $y_{1:100}$ is simulated with $a = 1/2$ and $\theta = \sqrt{0.1}$.

The problem of interest is the Bayesian static parameter estimation of $\theta$ from the above-mentioned model sequentially for each $n$ as the data arrives. {Our ultimate goal is to approximate
$\mathbb{E}_{\pi_{n}}[\varphi(\theta)]$, where $\varphi(\theta) = \theta$ and $\pi_{n}$
%=\pi_{n}(\theta|y_{1:n})$
is the posterior density of $\theta$, given $y_{1:n}$, induced by the HMM with no discretization bias. }
In this case, we are interested in the posterior mean of the model parameter $\theta$.

Given the approximation multi-index $\alpha^* = (m_x, m_t)$, we will
{estimate
$\mathbb{E}_{\pi_{n},\cI_{\alpha^*}}[\varphi(\theta)]=\mathbb{E}_{\pi_{n,\alpha^*}}[\varphi_{\alpha^*}(\theta)]$
to approximate $\mathbb{E}_{\pi_{n}}[\varphi(\theta)]$,
where $\pi_{n,\alpha^*}$ is the posterior distribution %given $y_{1:n}$ %$\pi_{n,\alpha^*}(\theta|y_{1:n})$
associated with the multi-index $\alpha^*$ and $\pi_{n}$ is the target posterior distribution.}
% $p(\theta|y_{1:n})$}.

We adopt the exponential Euler scheme developed in \cite{spde_disc} for discretizing the underlying hidden process. To be precise, at a multi-index $\alpha = (\alpha_x, \alpha_t)$, the above SPDE is solved with the first $K_\alpha = K_0 \times 2^{\alpha_x}$ eigenfunctions and $M_\alpha = M_0 \times 2^{\alpha_t}$ time steps as follows
\begin{equation}
	\label{expEuler}
	u_{\alpha,k,i+1} = e^{-\lambda_kh}u_{\alpha,k,i} + \frac{1-e^{-\lambda_kh}}{\lambda_k}a u_{\alpha,k,i} + r_{k,i}
\end{equation}
where $r_{k,i} \sim N\left(0,\frac{\theta^2(1-e^{-2\lambda_kh})}{2\lambda_k}\right)$ for $k=1,...,K_\alpha$ and $i=0,1,...,M_\alpha-1$. The time step-size $h = \delta/M_\alpha$ and $u_{\alpha,k,i}$ is the solution for the coefficient associated with the $k^{th}$ eigenfunction, $i^{th}$ time step and the discretization index $\alpha$.

The coupling of the $k_\alpha$ $(1 \leq k_\alpha \leq 4)$ discretized probability laws is constructed as follows. We start with the simulation of the most expensive random variable that corresponds to the multi-index $\alpha$. For simulations involving $\alpha_x - 1$, only the subset of the first $K_{\alpha-e_x}$ components are retained. For simulations involving $\alpha_t - 1$, $r_{k,i}$ in (\ref{expEuler}) is replaced by $\widehat{r}_{k,i} = e^{-\lambda_kh}\widetilde{r}_{k,2i} + \widetilde{r}_{k,2i+1}$ \cite{chernov2016multilevel} for $i=0,1,...,M_{\alpha-e_t}-1$, where $\{\widetilde{r}_{k,i}\}_{i=0}^{M_\alpha-1}$ are simulated with respect to the multi-index $\alpha$.

Assumption \ref{ass:mimc2} (b) was verified directly by estimating the quantity in Theorem \ref{thm:main}
using the empirical variance over 20 multi-increment estimators.
The values $\beta_x = 1$ and $\beta_t = 2$ were fit, for $\gamma_x=\gamma_t=1$,
which is consistent with the results in \cite{jklz}.
{We also assume $w_i = \beta_i/2$, as in \cite{jklz}, and this is verified numerically.
It is noted that assumption \ref{hyp:B} is likely not satisfied in this example, and so the numerical
results are testing the applicability of the method under weaker assumptions than provided by the theory.}

\blu{Following from the linear Gaussian form of (\ref{expEuler}), %the \cite{jklz}% by 
$u_{\alpha,k,i}$ is Gaussian. 
Since the observations are also linear and Gaussian,
the posterior on the state path is Gaussian can be computed exactly (for fixed parameters $\theta$) 
using the classical Kalman smoother \cite{law2015data}.
In fact, it can be computed without time discretization error, as shown in \cite{jklz}.
Following from standard identities for Gaussian random variables, 
its normalizing constant (the true marginal likelihood $p(y_{0:n}|\theta)$) 
can be computed exactly as well.
%since it fits in the framework of  state-space models. 
As a result, the true likelihood calculated from 
the Kalman smoother with high spatial resolution and no time discretization error
is used within standard MCMC to produce a ground truth, 
denoted by $\bbE(\varphi)$.
%benchmark for computing 
The MSE (denoted $\varepsilon^2$ in Proposition \ref{pro:mimcmc}) 
of the approximations is then computed as follows. 
%In particular, letting $\bbE(\varphi)$ 
%denote the ground truth, 
For a given estimator $\hat \varphi$ the MSE is estimated using 
$$
\frac1{R}\sum_{r=1}^{R} (\hat \varphi^{(r)} - \bbE(\varphi))^2 \, ,
$$ 
where $\hat \varphi^{(r)}$ is a realization of the estimator, i.e. using MCMC, 
MIPMCMC, or MISMC$^2$.}

\subsection{Results using PMCMC} \label{sec:numpmcmc}

\begin{figure}
	\centering
	\includegraphics[scale=0.7]{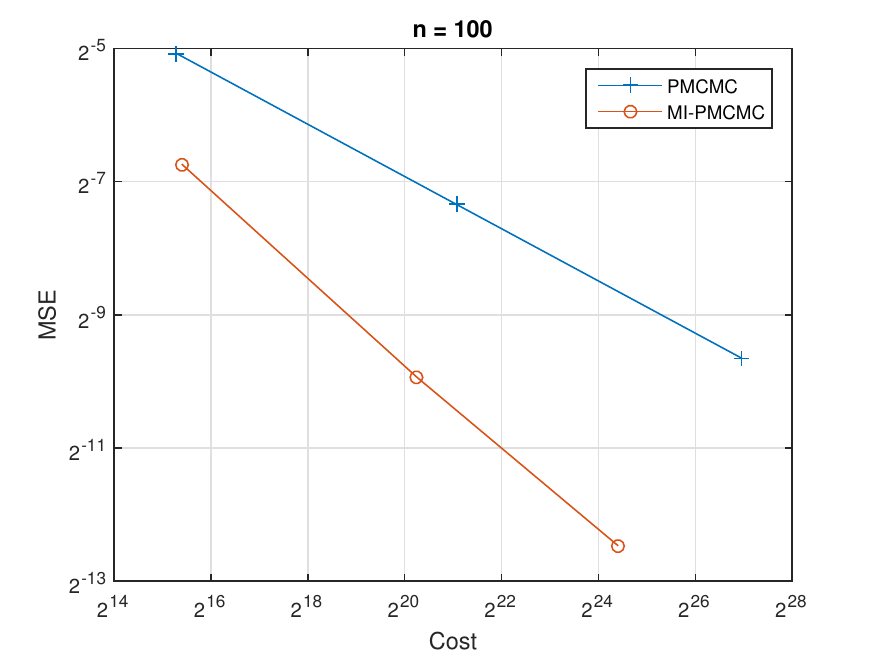}
	\caption{MSE v.s. Cost at time index $n = 100$ for the PMCMC method}
	\label{fg:PMCMC}
\end{figure}

We consider the estimation of the posterior mean of the model parameter $\theta$ 
in this section with $n$ fixed and $n = 100$. 
The proposed MIPMCMC method is implemented, as well as
a standard PMCMC at the finest discretization level.
\blu{The number of particles $N=500$ is fixed as well.}

%and 
%, 
%which is calculated here 
%using MCMC with a very large number of samples
%and exact marginal likelihood calculated with the Kalman filter, 
%with high spatial resolution and no time discretization error.
% is also implemented for comparison.

{Following the optimal choice of discretization $K = M^2$ as discussed in \cite{spde_disc}, the cost for a single realization is proportional to $M^3$. 
This results in the optimal cost for the ordinary PMCMC of $\cO(\varepsilon^{-5})$.
For the MIPMCMC method, we let $m_x = 2m_t \geq 2\text{log}(\varepsilon/2)$
and use the \blu{optimal $N_\alpha \propto \varepsilon^{-2}m_x2^{-\alpha_x-3\alpha_t/2}$, 
as mentioned in the proof of Proposition \ref{pro:mimcmc} and discussed further in \cite{jklz} and references therein.
The proportionality sign arises because the constants are unknown in the terms 
appearing in Assumption \ref{ass:mimc2} (b-c). 
In practice these are estimated along with the rates using simulations at lower levels.}
The cost is dominated by $\cO(\varepsilon^{-3})$ (where $3=\sum_{i=x,t}\gamma_i/w_i$ -- see discussion following Proposition \ref{pro:mimcmc}). }
Both algorithms are implemented for 20 runs and with the most precise discretization index $\alpha^* = (2,1), (4,2), (6,3)$. 
\blu{The MSE is then estimated using these $R=20$ realizations.}

The MSE vs cost plot is illustrated in Figure \ref{fg:PMCMC}.The cost rates are verified numerically as in Figure \ref{fg:PMCMC} and are consistent with %the discussions in
\cite{jklz}. The fitted rate is about $-5.1$ for the ordinary PMCMC method and $-3.1$ for MIPMCMC.
{It is noted that in the context of MLPMCMC for this example,
i.e. refining once in both $(x,t)$ at each level,
one will find $\gamma=3$, $\beta=2$, $\alpha=\beta/2$,
and $2+(\gamma-\beta)/\alpha = 3 (= \gamma/\alpha)$.
In other words, the rate is the same as we obtain here for MIPMCMC \cite{giles1}.}

\subsection{Results on {$\text{SMC}^2$}}
%The contour plot in Figure \ref{fg:contour} verified the

%\begin{figure}
%	\centering
%	\includegraphics[scale=0.7]{contour_color.eps}
%	\caption{Estimated variance of the multi-increments over a grid of multi-indices for $\text{SMC}^2$ method}
%	\label{fg:contour}
%\end{figure}

The proposed $\text{SMC}^2$ method as well as the ordinary $\text{SMC}^2$ method are implemented with the most precise discretization indices $\alpha^* = (2,1), (4,2), (6,3), (8,4)$ and as above the number of particles 
%$N$ 
is fixed at $N=500$. The proposed $\text{SMC}^2$ method is run with the optimal choice of $N_\alpha \propto \varepsilon^{-2}m_x2^{-\alpha_x-3\alpha_t/2}$ as discussed in \cite{jklz} and subsection \ref{sec:numpmcmc}. The ground truth in this case is calculated by the weighted average of the $\theta$ particles from the iterated batch importance sampling algorithm \cite{chopin2002sequential} with true likelihood increments derived from the Kalman techniques, which is used for computing the MSE of the approximations. Both algorithms are implemented for 20 runs and the MSE is estimated using these realizations.

{The same rate is expected as in subsection \ref{sec:numpmcmc}, under the same choices of
$(m_x,m_t)$ and $N_\alpha$.}
This is verified numerically, as illustrated in Figure \ref{fg:cost_error}, which displays the MSE vs cost plot at different time index $n \in \{50,65,80,100\}$. The fitted rate is about $-5.2$ for the ordinary $\text{SMC}^2$ and $-3$ for the multi-index $\text{SMC}^2$ method.

\begin{figure}
	\centering
	\subfigure{{\includegraphics[height=6cm,width=0.48\textwidth]{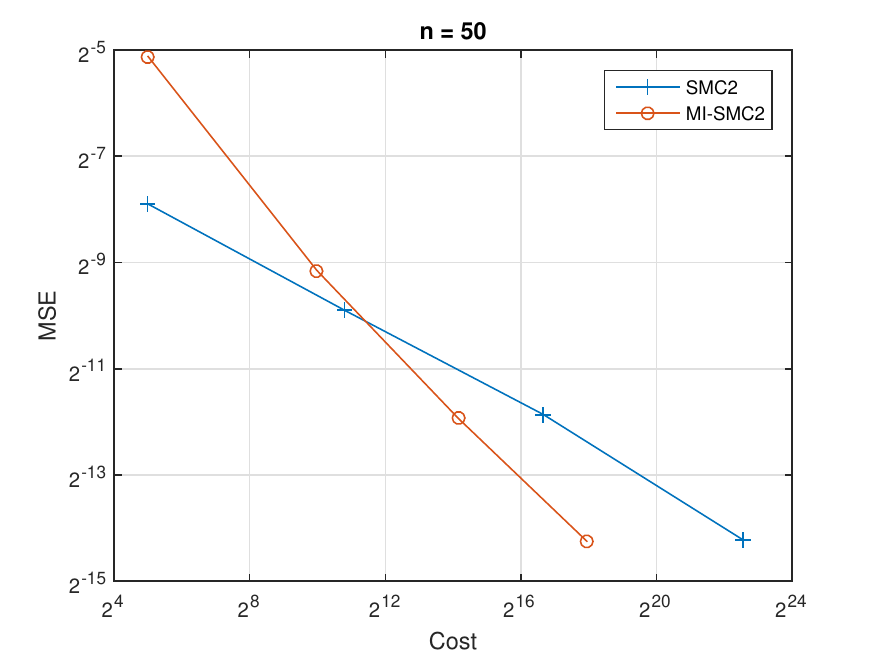}}} \,\,\,
	\subfigure{{\includegraphics[height=6cm,width=0.48\textwidth]{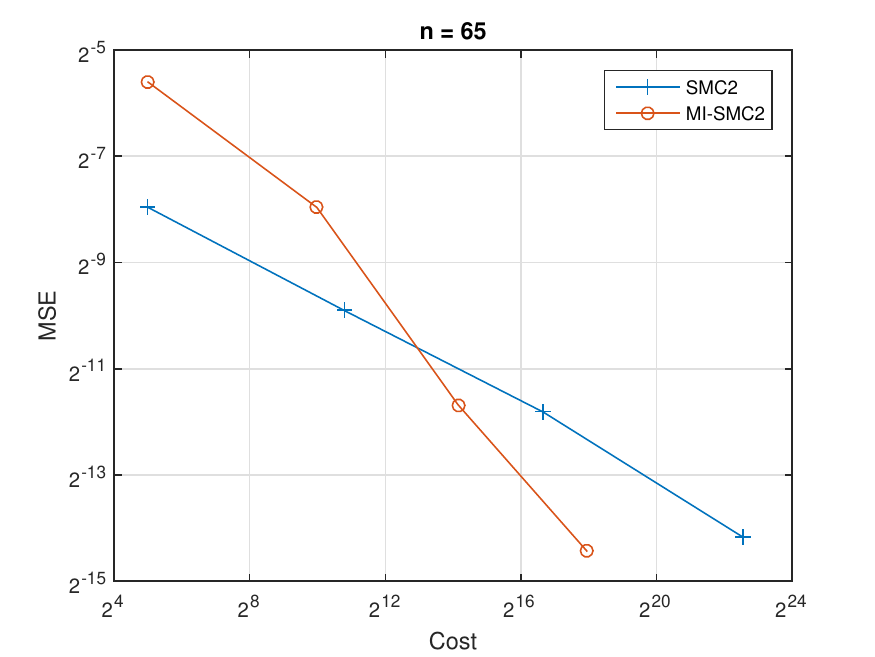}}} \\
	\subfigure{{\includegraphics[height=6cm,width=0.48\textwidth]{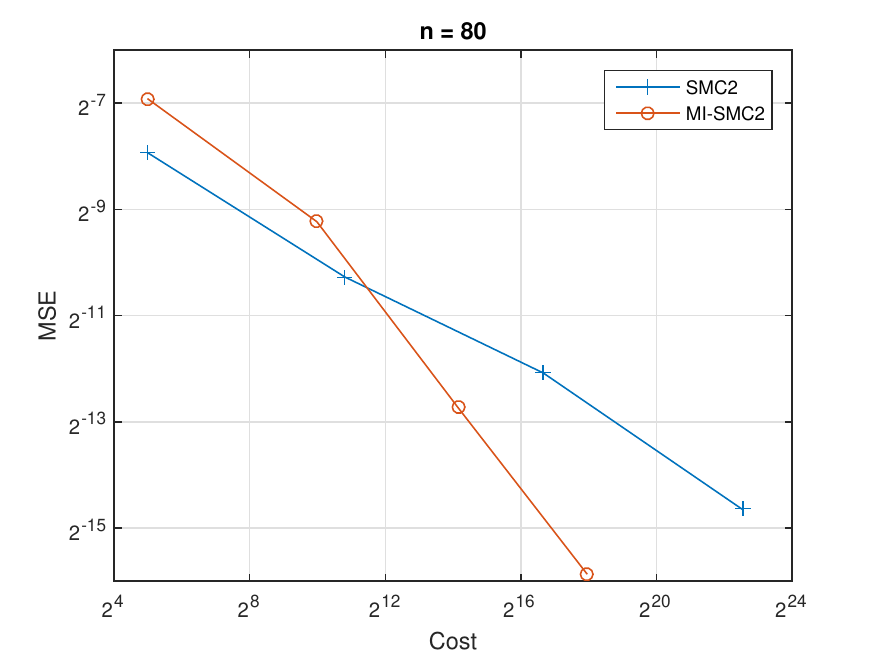}}} \,\,\,
	\subfigure{{\includegraphics[height=6cm,width=0.48\textwidth]{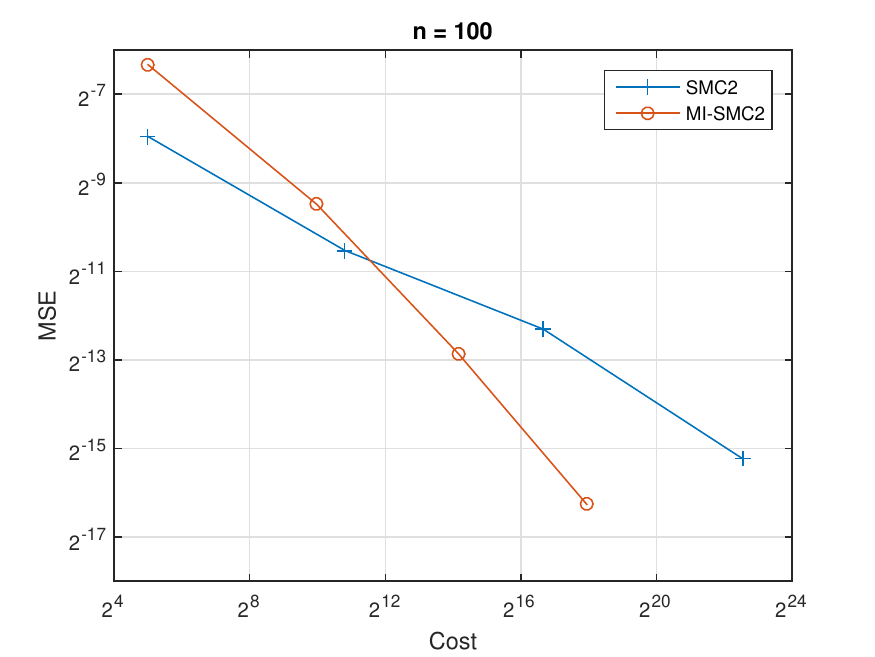}}}
	\caption{MSE v.s. Cost at different time index $n \in \{50,65,80,100\}$}
	\label{fg:cost_error}
\end{figure}

\section{Conclusion}

MISMC$^2$ and MIPMCMC are introduced. 
It is proven under strong assumptions 
that these methods achieve a better rate of complexity with respect to 
MSE than their single level counterparts, and this can even be canonical $1/$MSE. 
%SMC$^2$.
The algorithm is tested on a typical numerical example which may not satisfy the 
required assumptions, and the results are verified.
This makes us optimistic that theoretical results 
for the present algorithm can be obtained under weaker assumptions.
Another interesting direction for future research is further exploration of the method 
in practical scenarios and with optimal index sets.

\subsubsection*{Acknowledgements}
{We would like to thank Abdul-Lateef Haji-Ali for useful discussions relating to the material in this paper.}
%AJ was supported by an AcRF tier 2 grant: R-155-000-161-112. 
%AJ is affiliated with the Risk Management Institute, the Center for Quantitative Finance
%and the OR \& Analytics cluster at NUS. 
AJ was supported by a KAUST CRG4 grant ref: 2584 and KAUST baseline funding.
K.J.H.L. \& A.J. were supported by the U.S. Department of Energy, Office of Science, Office of Advanced Scientific Computing Research (ASCR),
under field work proposal number ERKJ333.
KJHL was additionally supported by 
The Alan Turing Institute under the EPSRC grant EP/N510129/1.
He was also funded in part by Oak Ridge National Laboratory Directed Research and Development Seed funding.

\appendix

\section{Main Proofs}

Let $(E,\mathcal{E})$ be a measurable space. The
supremum norm is written as $\|f\| = \sup_{u\in E}|f(u)|$.
 We will consider a non-negative operator
{$K : E \times \mathcal{E} \rightarrow \bbR_+$,  finite measure} $\mu$ on $(E,\mathcal{E})$  and a real-valued, measurable $f:E\rightarrow\mathbb{R}$ and define the operations:
%such that for each $u \in E$ the mapping $A \mapsto K(u, A)$ is a finite non-negative measure on $\mathcal{E}$ and for each $A \in \mathcal{E}$ the function $u \mapsto K(u, A)$ is measurable; the kernel $K$ is Markovian if $K(u, dv)$ is a probability measure for every $u \in E$.
% the PMMH step followed by
%the sampling of, for $i\in\{1,\dots,N_{\alpha}\}$ $X_p^{i,j}(1:k_{\alpha}),a_{p-1}^{i,j}$, $j\in\{1,\dots,N\}$ in Algorithm \ref{alg:smc2}, the iterate step.
%Denote the PMMH kernel as $\bar{M}_p:\mathsf{E}_{p-1}\times\mathcal{E}_{p-1}\rightarrow[0,1]$.
%
\begin{equation*}
    \mu K  : A \mapsto \int K(u, A) \, \mu(du)\ ;\quad
    K f :  u \mapsto \int f(v) \, K(u, dv).
\end{equation*}
We also write $\mu(f) = \int f(u) \mu(du)$.

Recall the definition of $\mathsf{E}_p$ in Section \ref{sec:smc2} and denote by $\mathcal{E}_p$ the associated $\sigma-$algebra.
Let $p\geq 1$ and denote by $M_p:\mathsf{E}_{p-1}\times\mathcal{E}_p\rightarrow[0,1]$ the Markov kernel which {is a compositon of
\begin{enumerate}
\item{A marginal PMCMC kernel $\bar{M}_p:\mathsf{E}_{p-1}\times\mathcal{E}_{p-1}\rightarrow[0,1]$, as in Algorithm \ref{alg:pmcmc} (ignoring $s$),}
\item{Followed by
the sampling of, for $i\in\{1,\dots,N_{\alpha}\}$ $X_p^{i,j}(1:k_{\alpha}),a_{p-1}^{i,j}$, $j\in\{1,\dots,N\}$ in Algorithm \ref{alg:smc2}, the iterate step.}
\end{enumerate}}
Denote by $\tilde{\eta}_0$ as the initial probability measure (on $(\mathsf{E}_0,\mathcal{E}_0)$)
of  $\theta^i$ and $X_0^{i,j}(1:k_{\alpha})$  $j\in\{1,\dots,N\}$ in Algorithm \ref{alg:smc2}, the initialization step.
Define the probability measure on $(\mathsf{E}_p\times\{1,\dots,N\}),\mathcal{E}_p\vee {2^{\{1,\dots,N\}}})$, $\eta_{p,\alpha}$:
$$
\eta_{p,\alpha}(d(u_{p,\alpha},s)) := \Big(\int_{\mathsf{E}_0\times\cdots\times\mathsf{E}_{p-1}} [\prod_{q=0}^pG_p(u_{p,\alpha})]\tilde{\eta}_0(du_{0,\alpha})[\prod_{q=1}^{p-1} M_q(u_{q-1,\alpha},du_{q,\alpha})]\bar{M}_p(u_{p-1,\alpha},du_{p,\alpha})\times
$$
$$
\mathbb{P}(s|u_{p,\alpha}) ds\Big)\Big/\Big(
 \int_{\mathsf{E}_0\times\cdots\times\mathsf{E}_{p-1}} [\prod_{q=0}^pG_p(u_{p,\alpha})]\tilde{\eta}_0(du_{0,\alpha})[\prod_{q=1}^{p-1} M_q(u_{q-1,\alpha},du_{q,\alpha})]\Big)
$$
where $ds$ is counting measure and $\mathbb{P}(s|u_{p,\alpha})$ is as \eqref{eq:s_def}.

Note that one can easily show that \eqref{eq:basic_idea} is equal to
$$
\sum_{i=1}^{k_{\alpha}'}(-1)^{|\alpha(k_{\alpha})-\alpha(2i)|}
\Bigg(\frac{
\eta_{n,\alpha}(\varphi_{\alpha(2i)}H_{2i,n,\alpha,\theta})
}{\eta_{n,\alpha}(H_{2i,n,\alpha,\theta})} -
\frac{
\eta_{n,\alpha}(\varphi_{\alpha(2i-1)}H_{2i-1,n,\alpha,\theta})
}{\eta_{n,\alpha}(H_{2i-1,n,\alpha,\theta})}\Bigg).
$$
%where for instance
%$$
%\eta_{n,\alpha}(\varphi_{\alpha(2i)}H_{2i,n,\alpha,\theta}) =  \int_{\mathsf{E}_n\times\{1,\dots,N\}}\varphi_{\alpha(2i)}(x_{0:n}^{s}(2i),\theta)H_{2i,n,\alpha,\theta}(x_{0:n}(1:k_{\alpha}))
%\eta_{n,\alpha}(d(u_{n,\alpha},s)).
%$$

Recall $\tau_{i,\alpha}=(-1)^{|\alpha(k_{\alpha})-\alpha(2i)|}$ and set, for each $\varphi,\alpha(i)$,
{
$$
\zeta_{i,n,\varphi}(x_{0:n}(1:k_{\alpha}),s,\theta)=\varphi_{\alpha(i)}(x_{0:n}^s(i),\theta)H_{i,n,\alpha,\theta}(x_{0:n}(1:k_{\alpha})).
$$}
Now set
\begin{eqnarray*}
\psi_{n,i,\alpha}^{N_{\alpha}} & := & \frac{\eta_{n,\alpha}^{N_{\alpha}}(\zeta_{2i-1,n,\varphi})}{\eta_{n,\alpha}^{N_{\alpha}}(H_{2i,n,\alpha,\theta})\eta_{n,\alpha}^{N_{\alpha}}(H_{2i-1,n,\alpha,\theta})} \\
\psi_{n,i,\alpha} & := & \frac{\eta_{n,\alpha}(\zeta_{2i-1,n,\varphi})}{\eta_{n,\alpha}(H_{2i,n,\alpha,\theta})\eta_{n,\alpha}(H_{2i-1,n,\alpha,\theta})} \\
\overline{\psi}_{n,i,\alpha}^{N_{\alpha}}  & := & \psi_{n,i,\alpha}^{N_{\alpha}}  - \psi_{n,i,\alpha}.
\end{eqnarray*}
In addition:
\begin{eqnarray*}
\Xi^{N_{\alpha}}_{n,\alpha,1} & := & \sum_{i=1}^{k_{\alpha}'}\tau_{i,\alpha}\Big[\eta_{n,\alpha}^{N_{\alpha}}(H_{2i,n,\alpha,\theta})^{-1}-\eta_{n,\alpha}(H_{2i,n,\alpha,\theta})^{-1}\Big][\eta_{n,\alpha}^{N_{\alpha}}-\eta_{n,\alpha}](\zeta_{2i,n,\varphi}-\zeta_{2i-1,n,\varphi}) \\
\Xi^{N_{\alpha}}_{n,\alpha,2} & := & \sum_{i=1}^{k_{\alpha}'}\tau_{i,\alpha}\overline{\psi}_{n,i,\alpha}^{N_{\alpha}}[\eta_{n,\alpha}^{N_{\alpha}}-\eta_{n,\alpha}](H_{2i,n,\alpha,\theta}-H_{2i-1,n,\alpha,\theta}) \\
\Xi^{N_{\alpha}}_{n,\alpha,3} & := & \sum_{i=1}^{k_{\alpha}'}\tau_{i,\alpha}\eta_{n,\alpha}(H_{2i,n,\alpha,\theta})^{-1}[\eta_{n,\alpha}^{N_{\alpha}}-\eta_{n,\alpha}](\zeta_{2i,n,\varphi}-\zeta_{2i-1,n,\varphi}) \\
\Xi^{N_{\alpha}}_{n,\alpha,4} & := & \sum_{i=1}^{k_{\alpha}'}\tau_{i,\alpha}\psi_{n,i,\alpha}[\eta_{n,\alpha}^{N_{\alpha}}-\eta_{n,\alpha}](H_{2i,n,\alpha,\theta}-H_{2i-1,n,\alpha,\theta}) \\
\Xi^{N_{\alpha}}_{n,\alpha,5} & := & \sum_{i=1}^{k_{\alpha}'}\tau_{i,\alpha}\Big[\eta_{n,\alpha}^{N_{\alpha}}(H_{2i,n,\alpha,\theta})^{-1}\eta_{n,\alpha}(H_{2i,n,\alpha,\theta})^{-1} - \eta_{n,\alpha}(H_{2i,n,\alpha,\theta})^{-2}\Big]\times \\ & &
\eta_{n,\alpha}(\zeta_{2i,n,\varphi}-\zeta_{2i-1,n,\varphi})[\eta_{n,\alpha}^{N_{\alpha}}-\eta_{n,\alpha}](H_{2i,n,\alpha,\theta})\\
\Xi^{N_{\alpha}}_{n,\alpha,6} & := & \sum_{i=1}^{k_{\alpha}'}\tau_{i,\alpha}\eta_{n,\alpha}(H_{2i,n,\alpha,\theta}-H_{2i-1,n,\alpha,\theta})\overline{\psi}_{n,i,\alpha}^{N_{\alpha}} \\
\Xi^{N_{\alpha}}_{n,\alpha,7} & := & \sum_{i=1}^{k_{\alpha}'}\tau_{i,\alpha}\eta_{n,\alpha}(H_{2i,n,\alpha,\theta})^{-2}\eta_{n,\alpha}(\zeta_{2i,n,\varphi}-\zeta_{2i-1,n,\varphi})[\eta_{n,\alpha}^{N_{\alpha}}-\eta_{n,\alpha}](H_{2i,n,\alpha,\theta}).
\end{eqnarray*}
{
\begin{lem}\label{lem:1}
Assume \ref{hyp:A}. Then for every $n\geq 0$, $\varphi:\mathbb{N}_0^d\times\mathsf{X}^{n+1}\times\Theta\rightarrow\mathbb{R}$ bounded and $\alpha\in\mathcal{I}$ we have that:
$$
\Delta\mathbb{E}^{N_{\alpha}}_{\pi_{n,\alpha}}[\varphi_{\alpha}(X_{0:n},\theta)]
-\Delta \mathbb{E}_{\pi_{n,\alpha}}[\varphi_{\alpha}(X_{0:n},\theta)] =
%\sum_{i=1}^{k_{\alpha}'}(-1)^{|\alpha(k_{\alpha})-\alpha(2i)|}
%\Bigg(\frac{
%\eta_{n,\alpha}^{N_{\alpha}}(\varphi_{\alpha(2i)}H_{2i,n,\alpha,\theta})
%}{\eta_{n,\alpha}^{N_{\alpha}}(H_{2i,n,\alpha,\theta})} -
%\frac{
%\eta_{n,\alpha}^{N_{\alpha}}(\varphi_{\alpha(2i-1)}H_{2i-1,n,\alpha,\theta})
%}{\eta_{n,\alpha}^{N_{\alpha}}(H_{2i-1,n,\alpha,\theta})}\Bigg) -
%$$
%$$
%\sum_{i=1}^{k_{\alpha}'}(-1)^{|\alpha(k_{\alpha})-\alpha(2i)|}
%\Bigg(\frac{
%\eta_{n,\alpha}(\varphi_{\alpha(2i)}H_{2i,n,\alpha,\theta})
%}{\eta_{n,\alpha}(H_{2i,n,\alpha,\theta})} -
%\frac{
%\eta_{n,\alpha}(\varphi_{\alpha(2i-1)}H_{2i-1,n,\alpha,\theta})
%}{\eta_{n,\alpha}(H_{2i-1,n,\alpha,\theta})}\Bigg) =
%$$
%$$
\sum_{j=1}^7(-1)^{j+1}\Xi^{N_{\alpha}}_{n,\alpha,j}.
$$
\end{lem}}
\begin{proof}
Follows by standard algebra. \ref{hyp:A} is only used to ensure the existence of all the associated quantities.
\end{proof}
{
\begin{prop}\label{prop:1}
Assume \ref{hyp:A}-\ref{hyp:B}.  Then for every $n\geq 0$, $\varphi:\mathbb{N}_0^d\times\mathsf{X}^{n+1}\times\Theta\rightarrow\mathbb{R}$ bounded and
$\alpha\in\mathcal{I}$, there exist a $C(\alpha(1:k_{\alpha}))$, with $\lim_{\min_{1\leq i\leq d}\alpha_i \rightarrow+\infty}C(\alpha(1:k_{\alpha}))=0$, such that for $j\in\{1,\dots,7\}$, $N_{\alpha}\geq 1$:
$$
\max\{|\mathbb{E}[\Xi^{N_{\alpha}}_{n,\alpha,j}]| ,\mathbb{E}[(\Xi^{N_{\alpha}}_{n,\alpha,j})^2]\} \leq \frac{C(\alpha(1:k_{\alpha}))}{N_{\alpha}}.
$$
\end{prop}}
\begin{proof}
We give the proofs in the case $j=1$ or $j=3$. All other cases are essentially the same and omitted.
{Throughout the proof $C(\alpha(1:k_{\alpha}))$ is a constant that depends on
$n\geq 0$, $\varphi:\mathbb{N}_0^d\times\mathsf{X}^{n+1}\times\Theta\rightarrow\mathbb{R}$, with $\lim_{\min_{1\leq i\leq d}\alpha_i \rightarrow+\infty}C(\alpha(1:k_{\alpha}))=0$.
The exact value of $C(\alpha(1:k_{\alpha}))$ may change from line to line, but the latter property holds.}

Set
$$
\kappa_{n,\alpha,1}(x_{0:n}(1:k_{\alpha}),s,\theta) = \sum_{i=1}^{k_{\alpha}'}\tau_{i,\alpha}\Big[\eta_{n,\alpha}^{N_{\alpha}}(H_{2i,n,\alpha,\theta})^{-1}-\eta_{n,\alpha}(H_{2i,n,\alpha,\theta})^{-1}\Big]\times
$$
$$
(\zeta_{2i,n,\varphi}(x_{0:n}(1:k_{\alpha}),s,\theta)-\zeta_{2i-1,n,\varphi}(x_{0:n}(1:k_{\alpha}),s,\theta)).
$$
Then
$$
\mathbb{E}[(\Xi^{N_{\alpha}}_{n,\alpha,1})^2] = \mathbb{E}\Big[[\eta_{n,\alpha}^{N_{\alpha}}-\eta_{n,\alpha}]\Big(\frac{\kappa_{n,\alpha,1}}{\|\kappa_{n,\alpha,1}\|}\Big)^2\|\kappa_{n,\alpha,1}\|^2\Big].
$$
Clearly, {applying \ref{hyp:B} to the term $\|\kappa_{n,\alpha,1}\|$} and {using that
\begin{equation}\label{eq:prf_ref}
[\eta_{n,\alpha}^{N_{\alpha}}-\eta_{n,\alpha}]\Big(\frac{\kappa_{n,\alpha,1}}{\|\kappa_{n,\alpha,1}\|}\Big) \leq 2
\end{equation}
} it follows that
$$
\mathbb{E}[(\Xi^{N_{\alpha}}_{n,\alpha,1})^2] \leq  C(\alpha(1:k_{\alpha}))\mathbb{E}\Big[\Big(\sum_{i=1}^{k_{\alpha}'}\Big|\eta_{n,\alpha}^{N_{\alpha}}(H_{2i,n,\alpha,\theta})^{-1}-\eta_{n,\alpha}(H_{2i,n,\alpha,\theta})^{-1}\Big|^2\Big)^2\Big] \, .
$$
Application of the Minkowski inequality yields
$$
\mathbb{E}[(\Xi^{N_{\alpha}}_{n,\alpha,1})^2] \leq  C(\alpha(1:k_{\alpha})) \Big(\sum_{i=1}^{k_{\alpha}'}\mathbb{E}\Big[\Big(\Big[\eta_{n,\alpha}^{N_{\alpha}}(H_{2i,n,\alpha,\theta})^{-1}-\eta_{n,\alpha}(H_{2i,n,\alpha,\theta})^{-1}\Big]\Big)^4\Big]^{1/2}\Big)^2.
$$
{
Note that the summand
$$
\mathbb{E}\Big[\Big(\Big[\eta_{n,\alpha}^{N_{\alpha}}(H_{2i,n,\alpha,\theta})^{-1}-\eta_{n,\alpha}(H_{2i,n,\alpha,\theta})^{-1}\Big]\Big)^4\Big]^{1/2} =
\mathbb{E}\Big[\Big(\Big[\frac{\eta_{n,\alpha}(H_{2i,n,\alpha,\theta})-\eta_{n,\alpha}^{N_{\alpha}}(H_{2i,n,\alpha,\theta})}{\eta_{n,\alpha}^{N_{\alpha}}(H_{2i,n,\alpha,\theta})\eta_{n,\alpha}(H_{2i,n,\alpha,\theta})}\Big]\Big)^4\Big]^{1/2}
$$
then applying \ref{hyp:A}
$$
\mathbb{E}\Big[\Big(\Big[\eta_{n,\alpha}^{N_{\alpha}}(H_{2i,n,\alpha,\theta})^{-1}-\eta_{n,\alpha}(H_{2i,n,\alpha,\theta})^{-1}\Big]\Big)^4\Big]^{1/2} \leq
C\mathbb{E}\Big[\Big(\eta_{n,\alpha}(H_{2i,n,\alpha,\theta})-\eta_{n,\alpha}^{N_{\alpha}}(H_{2i,n,\alpha,\theta})\Big)^4\Big]^{1/2}
$$
where $C<+\infty$ is a constant that does not depend upon $\alpha$.
Then applying \cite[Proposition 2.9]{delm:00} to the term on the r.h.s.~of the above equation,
yields that
\begin{equation}\label{eq:prf_ref1}
\mathbb{E}\Big[\|\kappa_{n,\alpha,1}\|^2\Big] \leq \frac{C(\alpha(1:k_{\alpha}))}{N_{\alpha}^2}
\end{equation}
and hence allows one to derive the upper-bound}
$$
\mathbb{E}[(\Xi^{N_{\alpha}}_{n,\alpha,1})^2] \leq \frac{C(\alpha(1:k_{\alpha}))}{N_{\alpha}^2}.
$$
For the bias, we have
$$
|\mathbb{E}[(\Xi^{N_{\alpha}}_{n,\alpha,1})]| \leq  \mathbb{E}\Big[\Big|[\eta_{n,\alpha}^{N_{\alpha}}-\eta_{n,\alpha}]\Big(\frac{\kappa_{n,\alpha,1}}{\|\kappa_{n,\alpha,1}\|}\Big)\Big|\|\kappa_{n,\alpha,1}\|\Big]
$$
it follows by {using \eqref{eq:prf_ref}}
$$
|\mathbb{E}[(\Xi^{N_{\alpha}}_{n,\alpha,1})]| \leq 2\mathbb{E}[\|\kappa_{n,\alpha,1}\|]
$$
{then using Jensen's inequality and \eqref{eq:prf_ref1},} we can conclude that
$$
|\mathbb{E}[(\Xi^{N_{\alpha}}_{n,\alpha,1})]| \leq \frac{C(\alpha(1:k_{\alpha}))}{N_{\alpha}}.
$$
%via the above calculations.
%Standard results in \cite[Chapter 7]{delm:04} and the above arguments give
%$$
%|\mathbb{E}[(\Xi^{N_{\alpha}}_{n,\alpha,1})| \leq \frac{C(\alpha(1:k_{\alpha}))}{N_{\alpha}}.
%$$

Set
$$
\kappa_{n,\alpha,3}(x_{0:n}(1:k_{\alpha}),s,\theta) = \sum_{i=1}^{k_{\alpha}'}\tau_{i,\alpha}\eta_{n,\alpha}(H_{2i,n,\alpha,\theta})^{-1}(\zeta_{2i,n,\varphi}(x_{0:n}(1:k_{\alpha}),s,\theta)-
$$
$$
\zeta_{2i-1,n,\varphi}(x_{0:n}(1:k_{\alpha}),s,\theta)).
$$
Then
$$
\mathbb{E}[(\Xi^{N_{\alpha}}_{n,\alpha,3})^2] = \mathbb{E}[[\eta_{n,\alpha}^{N_{\alpha}}-\eta_{n,\alpha}](\kappa_{n,\alpha,3})^2].
$$
{Applying \cite[Proposition 2.9]{delm:00} to the term on the r.h.s.}~yields
$$
\mathbb{E}[(\Xi^{N_{\alpha}}_{n,\alpha,3})^2] \leq \frac{\|\kappa_{n,\alpha,3}\|^2}{N_{\alpha}}.
$$
Application of \ref{hyp:B} gives
$$
\mathbb{E}[(\Xi^{N_{\alpha}}_{n,\alpha,3})^2] \leq \frac{C(\alpha(1:k_{\alpha}))}{N_{\alpha}}.
$$
For $|\mathbb{E}[(\Xi^{N_{\alpha}}_{n,\alpha,3})]|$ using a similar decomposition to \cite[eq.~(A.2)]{beskos} one can show that
$$
|\mathbb{E}[(\Xi^{N_{\alpha}}_{n,\alpha,3})]| \leq \frac{C(\alpha(1:k_{\alpha}))}{N_{\alpha}}.
$$
the proof is omitted as it is standard.
\end{proof}

 \bibliography{references}
\bibliographystyle{IJ4UQ_Bibliography_Style}

\end{document}